\documentclass[a4paper,twoside,bibtotoc,idxtotoc]{article}
\usepackage{amsmath,amssymb,amsfonts,amsthm} 
\usepackage{graphics,graphicx}                 
\usepackage{color}                    
\usepackage{hyperref,fancyhdr}                 
\usepackage[all]{xy}
\usepackage[applemac]{inputenc}
\usepackage[ngerman,english]{babel}
\usepackage{url}
\usepackage{abstract}

\newtheorem{theorem}{Theorem}[section]
\newtheorem{proposition}[theorem]{Proposition}
\newtheorem{corollary}[theorem]{Corollary}
\newtheorem{lemma}[theorem]{Lemma}

\newtheorem{theorem-example}[theorem]{Theorem $\backslash$ Example}
\renewenvironment{proof}{\begin{sloppypar}\noindent{\bf Proof}}{\hfill$\blacksquare$\end{sloppypar}}
\theoremstyle{definition}
\newtheorem{definition}[theorem]{Definition}
\newtheorem{example}[theorem]{Example}

\newtheorem{remark}[theorem]{Remark}

\newtheorem{open problem}[theorem]{Open Problem}
\newtheorem{construction}[theorem]{Construction}
\newtheorem{notation}[theorem]{Notation}
\DeclareMathOperator{\Aut}{Aut}
\DeclareMathOperator{\Inn}{Inn}

\DeclareMathOperator{\im}{im}
\DeclareMathOperator{\id}{id}

\DeclareMathOperator{\M}{M}

\DeclareMathOperator{\Hom}{Hom}

\DeclareMathOperator{\Diff}{Diff}

\DeclareMathOperator{\U}{U}

\DeclareMathOperator{\pr}{pr}

\DeclareMathOperator{\Ext}{Ext}

\DeclareMathOperator{\Det}{Det}

\begin{document}

\title{\bf{On Noncommutative Principal Bundles with Finite Abelian Structure Group}}

\author{Stefan Wagner\\Universit\"at Hamburg,Fachbereich Mathematik\\Bundesstra\ss e 55 (Geomatikum), Germany\\ \url{stefan.wagner@uni-hamburg.de}}

\maketitle

\begin{abstract}
\noindent
Let $\Lambda$ be a finite abelian group. A dynamical system with transformation group $\Lambda$ is a triple $(A,\Lambda,\alpha)$, consisting of a unital locally convex algebra $A$, the finite abelian group $\Lambda$ and a group homomorphism $\alpha:\Lambda\rightarrow\Aut(A)$, which induces an action of $\Lambda$ on $A$. In this paper we present a new, geometrically oriented approach to the noncommutative geometry of principal bundles with finite abelian structure group based on such dynamical systems. 

\noindent
\emph{Keywords}: Noncommutative differential geometry, dynamical systems, (trivial) principal bundles with finite abelian structure group, (trivial) noncommutative principal bundles with finite abelian structure group, graded algebras, crossed-product algebras, factor systems.

\noindent
\emph{MSC2010}: {46L87, 55R10 (primary), 37B05, 17A60 (secondary)}
\end{abstract}

\pagenumbering{arabic}

\section{Introduction}

The correspondence between geometric spaces and commutative algebras is a familiar and basic idea of algebraic geometry. Noncommutative topology started with the famous Gelfand--Naimark Theorems: Every commutative C*-algebra is the algebra of continuous functions vanishing at infinity on a locally compact space and vice versa. In particular, a noncommutative C*-algebra may be viewed as ``the algebra of continuous functions vanishing at infinity'' on a ``quantum space''. The aim of noncommutative geometry is to develop the basic concepts of topology, measure theory and differential geometry in algebraic terms and then to generalize the corresponding classical results to the setting of noncommutative algebras. The question whether there is a way to translate the geometric concept of a fibre bundle to noncommutative geometry is quite interesting in this context. In the case of vector bundles a refined version of the Theorem of Serre and Swan \cite{Swa62} gives the essential clue: The category of vector bundles over a manifold $M$ is equivalent to the category of finitely generated projective modules over $C^{\infty}(M)$. This observation leads to a notion of noncommutative vector bundles and is the connection between the topological K-theory based on vector bundles and algebraic K-theory. For principal bundles, free and proper actions offer a good candidate for a notion of noncommutative principal bundles (see e.\,g.  \cite{BHS07, BDH13, Ellwood00, Ph09, SchWa15a, SchWa15b}). In a purely algebraic setting, the well-established theory of Hopf--Galois extensions provides a wider framework comprising coactions of Hopf algebras (e.\,g. \cite{Haj04,LaSu05,Sch04}). We also would like to mention the related notion of noncommutative principal torus bundles proposed by Echterhoff, Nest, and Oyono-Oyono \cite{ENOO09} (see also \cite{HaMa10}), which relies on a noncommutative version of Green's Theorem. A similar geometric approach based on transformation groups was developed by the author in \cite{Wa11b, Wa11d}.

As is well-known from classical differential geometry, the relation between locally and globally defined objects is important for many constructions and applications. For example, a principal bundle $(P,M,G,q,\sigma)$ can be considered as a geometric object that is glued together from local pieces which are trivial, i.e., which are of the form $U\times G$ for some open subset $U$ of $M$. Thus, a natural step towards a geometrically oriented theory of ``noncommutative principal bundles" is to describe the \emph{trivial} objects first, i.e., to determine and to classify the ``trivial" noncommutative principal bundles. The case of the $n$-torus, i.e., $G=\mathbb{T}^n$, was treated in \cite{Wa11a}:

\begin{definition}\label{Trivial NCP T^n-Bundles}
(Trivial noncommutative principal torus bundles). A (smooth) dynamical system $(A,\mathbb{T}^n,\alpha)$ is called a (\emph{smooth}) \emph{trivial noncommutative principal $\mathbb{T}^n$-bundle}, if each isotypic component $A_{\bf k}$, ${\bf k}\in\mathbb{Z}^n$, contains an invertible element.
\end{definition}

\noindent
A hint for the quality of this definition for trivial noncommutative principal torus bundles is the observation that they have a natural counterpart in the theory of Hopf--Galois extensions: In fact, up to a suitable completion, they correspond to the so-called cleft $\mathbb{C}[\mathbb{Z}^n]$-comodule algebras; a notion that is close, although in general not equivalent, to the triviality of a principal bundle:

\begin{remark}
(Relation to cleft Hopf--Galois extensions). Let $G$ be a group. An algebra $A$ is a $\mathbb{C}[G]$-comodule algebra if and only if $A$ is a $G$-graded algebra (cf. \cite[Lemma 4.8]{BlMo}). Moreover, we conclude from \cite[Example 2.1.4]{Sch04} that a $G$-graded algebra $A=\bigoplus_{g\in G}A_g$ is a Hopf--Galois extension (of $A_{1_G}$) if and only if $A$ is strongly graded, i.e., if $A_{g}A_{g'}=A_{gg'}$ for all $g,g'\in G$. Now, a short calculation shows that a $\mathbb{C}[G]$-comodule algebra $A$ is cleft if and only if each grading space $A_g$ contains an invertible element. For more background on Hopf--Galois extensions, in particular for the definition of cleft extensions, we refer to \cite[Section 2.2]{Sch04}.
\end{remark}

\noindent
Unfortunately, if $\Lambda$ is a finite abelian group, then the class of cleft $\mathbb{C}[\Lambda]$-comodule algebras do not extend the classical geometry of trivial principal $\Lambda$-bundles. 

Our approach is inspired by the following observation: Without loss of generality we may assume that $\Lambda=C_{n_1}\times\cdots\times C_{n_k}$, where $k\in\mathbb{N}$ and each $C_{n_i}$ denotes the cyclic group of order $n_i$. A principal $\Lambda$-bundle $(P,M,\Lambda,q,\sigma)$ is trivial if and only if it admits a trivialization map. Such a trivialization map consists basically of $k$ (smooth) functions $f_i:P\rightarrow C_{n_i}$ satisfying $f_i^{n_i}=1$ and $f_i(\sigma(p,\lambda))=f_i(p)\cdot\lambda_i$ for all $p\in P$ and $\lambda=(\lambda_1,\ldots,\lambda_k)\in\Lambda$. From an algebraical point of view this condition means that each isotypic component of the (naturally) induced dynamical system $(C^{\infty}(P),\Lambda,\alpha)$ contains an invertible element of (some prescribed) finite order.

In Section 2 we present a geometrically oriented approach to the noncommutative geometry of trivial principal $C_n$-bundles based on dynamical systems of the form $(A,C_n,\alpha)$. We will in particular see that this approach extends the classical geometry of trivial principal $C_n$-bundles: If $A=C^{\infty}(P)$ for some manifold $P$, then we recover a trivial principal $C_n$-bundle and, conversely, each trivial principal bundle $(P,M,C_n,q,\sigma)$ gives rise to a trivial noncommutative principal $C_n$-bundle of the form $(C^{\infty}(P),C_n,\alpha)$. 

While in classical (commutative) differential geometry there exists up to isomorphy only one trivial principal $C_n$-bundle over a given manifold $M$, 
the situation completely changes in the noncommutative world. Indeed, Fourier decomposition shows that the underlying algebraic structure of a trivial noncommutative principal $C_n$-bundle $(A,C_n,\alpha)$ is the one of a so-called $(C_n,A^{C_n})$-crossed product algebra, which are described and classified in Appendix \ref{classcroprodalg}. Thus, given a unital algebra $B$ (which serves as a `` base"), the main goal of the third section is to provide a complete classification of $(C_n,B)$-crossed product algebras for which the associtated sequence (\ref{split}) of groups is split, i.e., to classify all ``algebraically" trivial noncommutative principal $C_n$-bundles with fixed point algebra $B$.

The goal of Section 4 is to extend the results of Section 2 to finite abelian groups (which are, up to an isomorphism, products of cyclic groups), i.e., to present a geometrically oriented approach to the noncommutative geometry of trivial principal bundles with finite abelian structure group based on dynamical systems of the form $(A,\Lambda,\alpha)$, where $\Lambda$ denotes a finite abelian group. Again, we will see that this approach extends the classical geometry of trivial principal $\Lambda$-bundles: If $A=C^{\infty}(P)$ for some manifold $P$, then we recover a trivial principal $\Lambda$-bundle and, conversely, each trivial principal bundle $(P,M,\Lambda,q,\sigma)$ gives rise to a trivial noncommutative principal $\Lambda$-bundle of the form $(C^{\infty}(P),\Lambda,\alpha)$. Moreover, we present a bunch of examples of trivial noncommutative principal $\Lambda$-bundles. In particular, given $n\in\mathbb{N}$, we will see that the matrix algebra $\M_n(\mathbb{C})$ carries the structure of a trivial noncommutative principal $C_n\times C_n$-bundle. Furthermore, the last part of Section 4 is dedicated to some ideas and problems concerning a classification theory of  `` algebraically" trivial noncommutative principal $\Lambda$-bundles (with a prescribed fixed point algebra $B$).

Again, let $\Lambda$ be a finite abelian group. Section 5 is finally devoted to a geometrically oriented approach to the noncommutative geometry of principal $\Lambda$-bundles. Since the freeness property of a group action is a local condition (cf. \cite[Remark 8.10]{Wa11d}), our main idea is inspired by the classical setting: Loosely speaking, a dynamical system $(A,\Lambda,\alpha)$ is called a noncommutative principal $\Lambda$-bundle, if it is ``locally" a trivial noncommutative principal $\Lambda$-bundle in the sense of Section \ref{TriPrinBundFinAbStrGr}. At this point a localization method for non-commutative algebras or, more generally, for dynamical systems enters the game (cf. \cite{Wa11d}). We prove that this approach extends the classical theory of principal $\Lambda$-bundles and present some noncommutative examples. In fact, we first show that each trivial noncommutative principal $\Lambda$-bundle carries the structure of a noncommutative principal $\Lambda$-bundle in its own right. We further show that examples of noncommutative principal $\Lambda$-bundles are provided by sections of algebra bundles with trivial noncommutative principal $\Lambda$-bundle as fibre, sections of algebra bundles which are pull-backs of principal $\Lambda$-bundles and sections of trivial equivariant algebra bundles.

Let $G$ be a group and $B$ be a unital algebra. A $(G,B)$-crossed product algebra is a $G$-graded unital algebra with $A_{1_G}=B$ and the additional property that each grading space contains an invertible element. For example, if $G$ is abelian and $B=\mathbb{C}$, then a $(G,B)$-crossed product algebra is the same as a $G$-quantum torus in the terminology of \cite{Ne07b}. In Appendix A we introduce a ``cohomology theory" for crossed product algebras, which is inspired by the classical cohomology theory of groups. The corresponding cohomology spaces are crucial for the classification of trivial noncommutative principal bundles with compact abelian structure group. A detailed discussion for the case $G=\mathbb{Z}^n$ can be found in \cite{Wa11a}. Finally, given a compact abelian group $G$, the last part of the appendix is devoted to a Landstad duality theorem for $C^*$-dynamical systems $(A,G,\alpha)$ with the property that each isotypic component contains an invertible element. We did not find such a result in the literature.

The present paper is part of a larger program with the intend of finding a geometric approach to noncommutative principal bundles (cf. \cite{Wa11a,Wa11b,Wa11d}). 

\section*{Preliminaries and Notations} All manifolds appearing in this paper are assumed to be finite-dimensional, paracompact, second countable and smooth. For the necessary background on (principal) bundles and vector bundles we refer to \cite{KoNo63}. All algebras are assumed to be associative and complex if not mentioned otherwise. Given an algebra $A$, we write $\Gamma_A:=\Hom_{\text{alg}}(A,\mathbb{C})\backslash\{0\}$ (with the topology of pointwise convergence on $A$) for the spectrum of $A$ and $\Aut(A)$ for the corresponding group of automorphisms in the category of $A$. Moreover, a dynamical system is a triple $(A,G,\alpha)$, consisting of a unital locally convex algebra $A$, a topological group $G$ and a group homomorphism $\alpha:G\rightarrow\Aut(A)$, which induces a continuous action of $G$ on $A$. We will also make use of concepts coming from classical group cohomology: If $G,H$ are groups and $p\in\mathbb{N}_0$, we say that a map $f:G^p\rightarrow H$ is normalized if
\[(\exists j)\,g_j={1_G}\,\,\,\Rightarrow\,\,\,f(g_1,\ldots,g_p)=1_H
\]and write $C^p(G,H)$ for the space of all normalized maps $G^p\rightarrow H$, the so called $p$-cochains. For a detailed background on group cohomology we refer to \cite[Chapter IV]{Ma95}. Finally, for $n\in\mathbb{N}$ we write 
\[C_n:=\{z\in\mathbb{C}^{\times}:\,z^n=1\}=\{\zeta^k:\,\zeta:=\exp(\frac{2\pi i}{n}),\,k=0,1,\ldots,n-1\}
\]for the cyclic subgroup of $\mathbb{T}$ of $n$-th roots of unity. 

\section{Trivial noncommutative principal $C_n$-bundles}

In this section we present a geometrically oriented approach to the noncommutative geometry of trivial principal $C_n$-bundles based on dynamical systems of the form $(A,C_n,\alpha)$. We will in particular see that this approach extends the classical geometry of trivial principal $C_n$-bundles: If $A=C^{\infty}(P)$ for some manifold $P$, then we recover a trivial principal $C_n$-bundle and, conversely, each trivial principal bundle $(P,M,C_n,q,\sigma)$ gives rise to a trivial noncommutative principal $C_n$-bundle of the form $(C^{\infty}(P),C_n,\alpha)$. 

\begin{notation} 
We recall that the map
\[\Psi:C_n\rightarrow\Hom_{\text{gr}}(C_n,\mathbb{T}),\,\,\,\Psi(\zeta^k)(\zeta):=\zeta^k
\]is an isomorphism of abelian groups. In the following we will identify the character group of $C_n$ with $C_n$ via the isomorphism $\Psi$. In particular, if $A$ is a unital locally convex algebra and $(A,C_n,\alpha)$ a dynamical system, then we write
\[A_k:=A_{\Psi(\zeta^k)}=\{a\in A:\alpha(\zeta).a=\Psi(\zeta^k)(\zeta)\cdot a=\zeta^k\cdot a\}
\]for the isotypic component corresponding to $\zeta^k$, $k=0,1,\ldots,n-1$.
\end{notation}

\begin{definition}\label{TNCPC_nB1}(Trivial noncommutative principal $C_n$-bundles).
A dynamical system $(A,C_n,\alpha)$ is called a \emph{trivial noncommutative principal $C_n$-bundle} if each isotypic component $A_k$ contains an invertible element $a_k$ satisfying $a_k^n=1_A$.
\end{definition}

\begin{remark}\label{TNCPC_nB2}
Let $(A,C_n,\alpha)$ be a dynamical system such that each isotypic component $A_k$, $k=0,1,\ldots,n-1$, contains an invertible element. Then the set
\[A^{\times}_h:=\bigcup_{0\leq k\leq n-1} A^{\times}_k
\]of homogeneous units is a subgroup of $A^{\times}$ containing $A_0^{\times}$ and we thus obtain the following short exact sequence
\begin{align}
1\longrightarrow A_0^{\times}\longrightarrow A^{\times}_h\stackrel{q}\longrightarrow C_n\longrightarrow 1\label{split}
\end{align}
of groups, where $q(a_k):=\zeta^k$. 
\end{remark}

\begin{proposition}\label{TNCPC_nB2,5}
A dynamical system $(A,C_n,\alpha)$ is a trivial noncommutative principal $C_n$-bundle if and only if the associtated sequence \emph{(}\ref{split}\emph{)} is split, i.e., if there is a group homomorphism $\sigma:C_n\rightarrow A^{\times}_h$ satisfying $q\circ\sigma=\id_{C_n}$. 
\end{proposition}

\begin{proof}
\,\,\,($``\Rightarrow"$) For this direction we first choose an invertible element $a_1\in A_1$ satisfying $a_1^n=1_A$. Then a short calculation shows that the map
\[\sigma:C_n\rightarrow A^{\times}_h,\,\,\,\sigma(\zeta^k):=(a_1)^k
\]defines a group homomorphism splitting the associated sequence (\ref{split}) of groups.

($``\Leftarrow"$) Let $\sigma:C_n\rightarrow A^{\times}_h$ be a group homomorphism splitting the associated sequence (\ref{split}) of groups and put $a_k:=\sigma(\zeta)^k$. Then $a_k\in A_k$ is invertible by definition and satisfies $a_k^n=\sigma(\zeta)^{kn}=\sigma(\zeta^{kn})=\sigma(1)=1_A$.
\end{proof}

\begin{remark}
Let $(A,C_n,\alpha)$ be a trivial noncommutative principal $C_n$-bundle and $\sigma:C_n\rightarrow A^{\times}_h$ a group homomorphism which splits (\ref{split}). Then the map 
\[A_0^{\times}\rtimes_S C_n\rightarrow A^{\times}_h,\,\,\,(a_0,\zeta^k)\mapsto a_0\sigma(\zeta)^k
\]is an isomorphism of groups, where the semidirect product is defined by the homomorphism
\[S:=C_{A_0^{\times}}\circ\sigma:C_n\rightarrow \Aut(A_0^{\times})
\]and $C_{A_0^{\times}}:A^{\times}_h\rightarrow \Aut(A^{\times}_0)$ denotes the conjugation action of $A^{\times}_h$ on $A^{\times}_0$. 
\end{remark}

\begin{lemma}\label{TNCPC_nB3}
Let $(A,C_n,\alpha)$ be a dynamical system and $B:=A^{C_n}$. Then the following statements are equivalent:
\begin{itemize}
\item[\emph{(a)}]
There exist invertible elements $a_1\in A_1$ and $b_1\in B$ such that $a_1^n=b_1^n$.
\item[\emph{(b)}]
For each $k=0,1,\ldots,n-1$ there exist invertible element $a_k\in A_k$ and $b_k\in B$ such that $a_k^n=b_k^n$.
\end{itemize}
\end{lemma}

\begin{proof}
\,\,\,(a) $\Rightarrow$ (b): For $k=0,1,\ldots,n-1$ we simply put $a_k:=a_1^k$ and $b_k:=b_1^k$. Then $a_k\in A_k$ and $b_k\in B$ are both invertible and we have
\[a_k^n=a_1^{kn}=(a_1^n)^k=(b_1^n)^k=(b_1^k)^n=b_k^n.
\]

(b) $\Rightarrow$ (a): This direction is obvious.
\end{proof}

\begin{proposition}\label{TNCPC_nB4}
Let $(A,C_n,\alpha)$ be a dynamical system such that $B:=A^{C_n}$ is a central subalgebra of $A$. Then $(A,C_n,\alpha)$ is a trivial noncommutative principal $C_n$-bundle if and only if $(A,C_n,\alpha)$ satisfies one of the equivalent conditions of Lemma \ref{TNCPC_nB3}.
\end{proposition}

\begin{proof}
\,\,\,($``\Rightarrow"$) This direction immediately follows from Definition \ref{TNCPC_nB1}.

($``\Leftarrow"$) For the other direction we first choose invertible elements $a_1\in A_1$ and $b_1\in B$ such that $a_1^n=b_1^n$. Then a short calculation shows that the map
\[\sigma:C_n\rightarrow A^{\times}_h,\,\,\,\sigma(\zeta^k):=(a_1b_1^{-1})^k
\]is a group homomorphism which splits the associated sequence (\ref{split}) of groups.
\end{proof}


\begin{lemma}\label{TNCPC_nB5}
If $A$ is a unital locally convex algebra and $(A,C_n,\alpha)$ a dynamical system such that each isotypic component $A_k$ contains an invertible element, then the induced map
\[\sigma:\Gamma_A\times C_n\rightarrow\Gamma_A,\,\,\,\chi.\zeta^k:=\sigma(\chi,\zeta^k):=\chi\circ\alpha(\zeta^k)
\]defines a free action of $C_n$ on the spectrum $\Gamma_A$ of $A$. 
\end{lemma}

\begin{proof}
\,\,\,An easy observation shows that $\sigma$ defines an action of $C_n$ on the spectrum $\Gamma_A$. The crucial part is to verify the freeness of the map $\sigma$, i.e., to show that the stabilizer of each element of $\Gamma_A$ is trivial: For this, we first choose an invertible element $a_1\in A_1$. Now, let $\chi\in\Gamma_A$ and $0\leq l\leq n-1$ such that $\chi.\zeta^l=\chi\circ\alpha(\zeta^l)=\chi$. Then 
\[(\chi\circ\alpha(\zeta^l))(a_1)=\chi(\alpha(\zeta^l).a_1)=\zeta^{l}\cdot\chi(a_1)=\chi(a_1)
\]implies that $\zeta^{l}=1$ and thus in turn that $l=0$, which proves the freeness of the map $\sigma$.
\end{proof}

\begin{remark}\label{TNCPC_nB6}
If $P$ is a manifold, $p\in P$ and $\delta_p$ the corresponding point evaluation map on $C^{\infty}(P)$, then there is a unique smooth structure on the spectrum $\Gamma_{C^{\infty}(P)}$ of $C^{\infty}(P)$ for which the map
\[\Phi:P\rightarrow \Gamma_{C^{\infty}(P)},\,\,\,p\mapsto\delta_p
\]becomes a diffeomorphism. A proof of this statement can be found in \cite[Lemma 2.5]{Wa11c}.
\end{remark}

\begin{proposition}\label{TNCPC_nB7}
Let $P$ be a manifold and $(C^{\infty}(P),C_n,\alpha)$ a dynamical system such that each isotypic component $C^{\infty}(P)_k$ contains an invertible element. Then the map 
\[\sigma:P\times C_n\rightarrow P,\,\,\,(\delta_p,\zeta^k)\mapsto\delta_p\circ\alpha(\zeta^k),
\]where we have identified $P$ with the set of characters via the map $\Phi$ from Remark \ref{TNCPC_nB6}, is smooth and defines a free and proper action of $C_n$ on the manifold $P$. In particular, we obtain a principal bundle $(P,P/C_n,C_n,\pr,\sigma)$, where $\pr:P\rightarrow P/C_n$ denotes the canonical orbit map.
\end{proposition}

\begin{proof}
\,\,\,Since the group $C_n$ is finite, the map $\sigma$ is obviously smooth and proper. Its freeness follows from Lemma \ref{TNCPC_nB5}. Therefore, the Quotient Theorem implies that we obtain a principal bundle $(P,P/C_n,C_n,\pr,\sigma)$ (cf. \cite[Kapitel VIII, Satz 21.6]{To00}).
\end{proof}


\begin{theorem}\label{TNCPC_nB8}\emph{(}Trivial principal $C_n$-bundles\emph{)}.
Let $P$ be a manifold. Then the following assertions hold:
\begin{itemize}
\item[\emph{(a)}]
If $(C^{\infty}(P),C_n,\alpha)$ is a trivial noncommutative principal $C_n$-bundle, then the corresponding principal bundle $(P,P/C_n,C_n,\pr,\sigma)$ of Proposition \ref{TNCPC_nB7} is trivial.
\item[\emph{(b)}]
Conversely, if $(P,M,C_n,q,\sigma)$ is a trivial principal $C_n$-bundle, then the corresponding dynamical system $(C^{\infty}(P),C_n,\alpha)$ defined by
\[\alpha: C_n\times C^{\infty}(P)\rightarrow C^{\infty}(P),\,\,\,\alpha(\zeta^k,f)(p):=(\zeta^k.f)(p):=f(\sigma(p,\zeta^k)),
\]is a trivial noncommutative principal $C_n$-bundle.
\end{itemize}
\end{theorem}

\begin{proof}
\,\,\,(a) If $(C^{\infty}(P),C_n,\alpha)$ is a trivial noncommutative principal $C_n$-bundle, then we may choose an invertible element \mbox{$f\in C^{\infty}(P)_1$} satisfying $f^n=1$ from which we conclude that $\im(f)=C_n$. In particular, the map
\[\varphi:P\rightarrow P/C_n\times C_n,\,\,\,p\mapsto(\pr(p),f(p))
\]defines an equivalence of principal $C_n$-bundles over $P/C_n$ implying that the principal bundle $(P,P/C_n,C_n,\pr,\sigma)$ of Proposition \ref{TNCPC_nB7} is trivial.

(b) Conversely, let $(P,M,C_n,q,\sigma)$ be a trivial principal $C_n$-bundle and
\[\varphi:P\rightarrow M\times C_n,\,\,\,p\mapsto(q(p),f(p))
\]be an equivalence of principal $C_n$-bundles over $M$. We first note that the function $f\in C^{\infty}(P)$ is invertible. Furthermore, the $C_n$-equivariance of $\varphi$ implies that $f\in C^{\infty}(P)_1$. Hence, $f\in C^{\infty}(P)_1$ is invertible and satisfies $f^n=1$. We thus conclude that $(C^{\infty}(P),C_n,\alpha)$ is a trivial noncommutative principal $C_n$-bundle.
\end{proof}

\section{Classification of trivial noncommutative\\principal $C_n$-bundles}\label{classoftrivnonC_n}

While in classical (commutative) differential geometry there exists up to isomorphy only one trivial principal $C_n$-bundle over a given manifold $M$, 
the situation completely changes in the noncommutative world. Indeed, Fourier decomposition shows that the underlying algebraic structure of a trivial noncommutative principal $C_n$-bundle $(A,C_n,\alpha)$ is the one of a so-called $(C_n,A^{C_n})$-crossed product algebra, which are described and classified in Appendix\ref{classcroprodalg}. Thus, given a unital algebra $B$ (which serves as a `` base"), the main goal of this section is to provide a complete classification of $(C_n,B)$-crossed product algebras for which the associtated sequence (\ref{split}) of groups is split, i.e., to classify all ``algebraically" trivial noncommutative principal $C_n$-bundles with fixed point algebra $B$.

\begin{definition}\label{equivalence relation}
Let $B$ be a unital algebra. We say that two group homomorphisms $S,S':C_n\rightarrow\Aut(B)$ are equivalent and write in that case $S\sim S'$ if there exists an element $h\in C^1(C_n,B^{\times})$ satisfying the following two conditions:
\begin{itemize}
\item[(i)]
We have $S'=(C_B\circ h)\cdot S$, where $C_B:B^{\times}\rightarrow\Aut(B)$ denotes the conjugation action of $B^{\times}$ on $B$.
\item[(ii)]
The class $[d_Sh]\in H^2(C_n,Z(B)^{\times})_S$ (cf. \cite[Chapter IV, Section 4]{Ma95} for the corresponding definition) vanishes, where
\[d_Sh(\zeta^k,\zeta^l):=h(\zeta^k)S(\zeta^k)(h(\zeta^l))h(\zeta^{k+l})^{-1}.
\] 
\end{itemize}
\end{definition}

\begin{remark}\label{ClassTriNonCommPB1}
To see that the 2-cochain $d_Sh$ in Definition \ref{equivalence relation} (ii) is actually a 2-cocycle, we just have to note that condition (i) is equivalent to $\im(d_Sh)\subseteq Z(B)^{\times}$, which in turn implies that $d_Sh\in Z^2(C_n,Z(B)^{\times})_S$ (cf. \cite[Remark 2.20 (b)]{Ne07a}).
\end{remark}

\begin{lemma}\label{ClassTriNonCommPB2}
Let $B$ be a unital algebra. Then $\sim$ defines an equivalence relation on the set $\Hom_{\emph{\text{gr}}}(C_n,\Aut(B))$.
\end{lemma}

\begin{proof}
\,\,\,For the proof we have to check that $\sim$ is reflexive, symmetric and transitive:
\begin{itemize}
\item ``Reflexivity": We just take $h\equiv 1_B$.
\item ``Symmetrie": We take $h':=h^{-1}$. Then a short calculation shows that $d_{S'}h'=d_Sh$.
\item ``Transitivity": If $S\sim S'$ with $h\in C^1(C_n,B^{\times})$ and $S'\sim S''$ with \mbox{$h'\in C^1(C_n,B^{\times})$}, then we easily conclude $S\sim S''$ with $h'\cdot h\in C^1(C_n,B^{\times})$. In fact, another short calculation yields to $d_S(h'\cdot h)=d_{S'}h'+d_Sh$.
\end{itemize}
\end{proof}

\begin{definition}(Set of equivalence classes).\label{ClassTriNonCommPB3}
Let $B$ be a unital algebra. We write $\Ext(C_n,B)_{\text{split}}$ for the set of all equivalence classes of $(C_n,B)$-crossed product algebras for which the associtated sequence (\ref{split}) of groups is split.
\end{definition}

\begin{proposition}\label{ClassTriNonCommPB9}
Let $B$ be a unital algebra and $S:C_n\rightarrow\Aut(B)$ a group homomorphism. Then the $(C_n,B)$-crossed product algebra $A_S:=A_{(S,{\bf 1})}$ of Construction \ref{realization of TNCT^B from factor systems I} defines an element of $\Ext(C_n,B)_{\emph{\text{split}}}$.
\end{proposition}

\begin{proof}
\,\,\,We just have to note that the map $\sigma:C_n\rightarrow A^{\times}_h$, $\sigma(\zeta^k):=v_{\zeta^k}$ defines a group homomorphism splitting the associated sequence (\ref{split}) of groups.
\end{proof}

\begin{theorem}\label{ClassTriNonCommPB4}
Let $B$ be a unital algebra. Then the map
\[\Phi:\Hom_{\emph{\text{gr}}}(C_n,\Aut(B))/\sim\,\rightarrow \Ext(C_n,B)_{\emph{\text{split}}},\,\,\,[S]\mapsto[A_S]
\]is a well-defined bijection.
\end{theorem}

\begin{proof}
\,\,\,We divide the proof of this theorem into three parts:

(i) That the map $\Phi$ is well-defined is a consequence of Proposition \ref{ClassTriNonCommPB9} and \cite[Remark 2.20 (b)]{Ne07a}: In fact, if $S\sim S'$ with $h\in C^1(C_n,B^{\times})$, then we get
\[[A_S]=[A_{(S',d_Sh)}]=[d_Sh].[A_{S'}]=[A_{S'}]
\](cf. Theorem \ref{class of Z-kernels II}).

(ii) Next, we show that $\Phi$ is surjective: For this, let $[A]\in\Ext(C_n,B)_{\text{split}}$ and choose a group homomorphism $\sigma:C_n\rightarrow A^{\times}_h$ splitting the associated sequence (\ref{split}) of groups. Then $S:=C_B\circ\sigma:C_n\rightarrow\Aut(B)$ defines a group homomorphism satisfying $\Phi([S])=[A_S]=[A]$.

(iii) To see that $\Phi$ is injective, we choose $S,S'\in\Hom_{\text{gr}}(C_n,\Aut(B))$ with $[A_S]=[A_{S'}]$. Then there exists $h\in C^1(C_n,B^{\times})$ with $S'=(C_B\circ h)\cdot S$ what is equivalent to $\im(d_Sh)\subseteq Z(B)^{\times}$, which in turn implies that $d_Sh\in Z^2(C_n,Z(B)^{\times})_S$ (cf. \cite[Remark 2.20 (b)]{Ne07a}). Now, we conclude from
\[[A_S]=[A_{(S',d_Sh)}]=[d_Sh].[A_{S'}]=[A_{S'}]
\]that $[d_Sh]\in H^2(C_n,Z(B)^{\times})_S$ vanishes. Therefore, we obtain $S\sim S'$ (with $h\in C^1(C_n,B^{\times})$).
\end{proof}

\begin{corollary}\label{ClassTriNonCommPB5}
If $B$ is commutative, then the map 
\[\Phi:\Hom_{\emph{\text{gr}}}(C_n,\Aut(B))\rightarrow \Ext(C_n,B)_{\emph{\text{split}}},\,\,\,S\mapsto [A_S]
\]is a well-defined bijection.
\end{corollary}

\begin{proof}
\,\,\,The assertion immediately follows from Theorem \ref{ClassTriNonCommPB4} and the fact that in this case we have $S\sim S'$ if and only if $S=S'$.
\end{proof}

\begin{remark}\label{ClassTriNonCommPB10}
Let $\Lambda$ be an abelian group and $S:C_n\rightarrow\Aut(\Lambda)$ be a group homomorphism. Further, let $N\Lambda$ be the subgroup of $\Lambda$ generated by the elements of the form $\lambda+S(\zeta).\lambda+\ldots+S(\zeta^{n-1}).\lambda$ for $\lambda\in\Lambda$, i.e.,
\[N\Lambda:=\left\langle\lambda+S(\zeta).\lambda+\ldots+S(\zeta^{n-1}).\lambda:\,\lambda\in\Lambda\right\rangle.
\]Then we have
\[H^2(C_n,\Lambda)_S\cong \Lambda^{C_n}/N\Lambda.
\]In particular, if $S$ is trivial, then $H^2(C_n,\Lambda)\cong\Lambda/n\Lambda$ and thus $H^2(C_n,\Lambda)$ is trivial for divisible $\Lambda$. A nice reference for the previous discussion is \cite[Chapter IV, Section 7]{Ma95}.
\end{remark}

\begin{example}\label{ClassTriNonCommPB6}
If $B=\mathbb{C}$, then $\Aut(B)=\{\id_{\mathbb{C}}\}$. We thus conclude from Corollary \ref{ClassTriNonCommPB5} and Remark \ref{ClassTriNonCommPB10} that $\Ext(C_n,\mathbb{C})_{\text{split}}$ consists of a single element which is geometrically realized as the trivial principal $C_n$-bundle over a single point $\{\ast\}$, which algebraically corresponds to the group algebra $\mathbb{C}[C_n]$.
\end{example}

\begin{example}\label{ClassTriNonCommPB7}
If $B=C^{\infty}(M)$ for some manifold $M$, then $\Aut(B)\cong\Diff(M)$. We thus conclude from Corollary \ref{ClassTriNonCommPB5} that $\Ext(C_n,C^{\infty}(M))_{\text{split}}$ is classified by diffeomorphisms of the manifold $M$ of finite order $n$. In particular, the trivial principal $C_n$-bundle over $M$ corresponds to the trivial diffeomorphism, i.e., to $S={\bf 1}$.
\end{example}

\begin{example}\label{ClassTriNonCommPB8}
Let $B=\M_m(\mathbb{C})$ for some $m\in\mathbb{N}$. Then $Z(\M_m(\mathbb{C}))^{\times}\cong\mathbb{C}^{\times}$ and according to the well-known Skolem--Noether Theorem each automorphism of $\M_m(\mathbb{C})$ is inner. In particular, each group homomorphism $S':C_n\rightarrow\Inn(\M_m(\mathbb{C}))$ is equivalent to $S={\bf 1}$, since $H^2(C_n,\mathbb{C}^{\times})$ is trivial (cf. Remark \ref{ClassTriNonCommPB10}). From this we conclude that $\Ext(C_n,\M_m(\mathbb{C}))_{\text{split}}$ consists of a single element, which is realized by the group algebra $\M_m(\mathbb{C})[C_n]$ (with multiplication given by the usual convolution product).
\end{example}

\section{Trivial noncommutative principal bundles with finite abelian structure group}\label{TriPrinBundFinAbStrGr}

In the following let $\Lambda$ be a finite abelian group. The goal of this section is to present a geometrically oriented approach to the noncommutative geometry of trivial principal $\Lambda$-bundles based on dynamical systems of the form $(A,\Lambda,\alpha)$. Again, we will see that this approach extends the classical geometry of trivial principal $\Lambda$-bundles: If $A=C^{\infty}(P)$ for some manifold $P$, then we recover a trivial principal $\Lambda$-bundle and, conversely, each trivial principal bundle $(P,M,\Lambda,q,\sigma)$ gives rise to a trivial noncommutative principal $\Lambda$-bundle of the form $(C^{\infty}(P),\Lambda,\alpha)$.

\begin{notation} 
Let $\Lambda$ be a finite abelian group and let $\{\lambda_1,\ldots,\lambda_k\}$ be a set of generators of $\Lambda$ (having order $n_1,\ldots, n_k\in \mathbb{N}$) of minimal cardinality. To each such generator $\lambda_i$ 
we associate an element $\widehat{\lambda}_i$ of its dual group $\widehat{\Lambda}:=\Hom_{\text{gr}}(\Lambda,\mathbb{T})$ by defining
\[\widehat{\lambda}_i(\lambda_j):=\begin{cases}
\zeta_i &\text{for}\,\,\,i=j\\
1 &\text{otherwise}.
\end{cases}
\]Here, the complex numbers $\zeta_i$ are defined for each $i=1,\ldots,k$ through
\[C_{n_i}:=\{z\in\mathbb{C}^{\times}:\,z^{n_i}=1\}=\left\{\zeta_i^l:\,\zeta_i:=\exp(\frac{2\pi i}{n_i}),\,l=0,1,\ldots,n_i-1\right\}.
\]
\end{notation}

\begin{definition}\label{C_nC_mI}(Trivial noncommutative principal $\Lambda$-bundles).
A dynamical system $(A,\Lambda,\alpha)$ is called a \emph{trivial noncommutative principal $\Lambda$-bundle} if there exists a set of generators $\{\lambda_1,\ldots,\lambda_k\}$ of $\Lambda$ (having order $n_1,\ldots, n_k\in \mathbb{N}$) of minimal cardinality such that each associated isotypic component
\[A_{\widehat{\lambda}_i}:=\{a\in A:\,(\forall\lambda\in\Lambda): \alpha(\lambda,a)=\widehat{\lambda}_i(\lambda)\cdot a\}
\]contains an invertible element $a_i$ satisfying $a_i^{n_i}=1_A$.
\end{definition}

\begin{remark}
Note that if $\Lambda$ is cyclic, i.e., if $\Lambda=C_n$ for some $n\in\mathbb{N}$, then Definition \ref{C_nC_mI} coincides with Definition \ref{TNCPC_nB1}.
\end{remark}


\begin{remark}\label{C_nC_mII}
Let $(A,\Lambda,\alpha)$ be a dynamical system such that each isotypic component contains an invertible element. Then the set $A^{\times}_h$ of homogeneous units is a subgroup of $A^{\times}$ containing $A_0^{\times}$ and we thus obtain the following short exact sequence
\begin{align}
1\longrightarrow A_0^{\times}\longrightarrow A^{\times}_h\stackrel{q}\longrightarrow\widehat{\Lambda}\longrightarrow 1\label{split Lambda}
\end{align}
of groups, where $q(a_{\widehat{\lambda}}):=\widehat{\lambda}$ (whenever $\widehat{\lambda}\in\widehat{\Lambda}$ and $a_{\widehat{\lambda}}\in A_{\widehat{\lambda}}$). If now $\{\lambda_1,\ldots,\lambda_k\}$ is a set of generators of $\Lambda$ (having order $n_1,\ldots, n_k\in \mathbb{N}$) of minimal cardinality, then each of the natural inclusion maps $\iota_i:\widehat{\Lambda}_i\rightarrow\widehat{\Lambda}$ gives rise to a pull back extension 
\begin{align}
E_i:1\longrightarrow A_0^{\times}\longrightarrow A^{\times}_{h,i}\longrightarrow \widehat{\Lambda}_i\longrightarrow 1\label{pull back}
\end{align}
of (\ref{split Lambda}). Here, $\widehat{\Lambda}_i$ denotes the subgroup of $\widehat{\Lambda}$ generated by $\widehat{\lambda}_i$, i.e., $\widehat{\Lambda}_i:=\langle\widehat{\lambda}_i\rangle_{\text{gr}}$ (note that $\widehat{\lambda}_i$ has by definition order $n_i\in\mathbb{N}$) and
\[A^{\times}_{h,i}:=\bigcup_{0\leq l\leq n_i-1} A^{\times}_{\widehat{\lambda}_i^l}.
\]
\end{remark}

\begin{proposition}\label{C_nC_mIII}
A dynamical system $(A,\Lambda,\alpha)$ is a trivial noncommutative principal $\Lambda$-bundle if and only if each pull back extension $E_i$ \emph{(}\ref{pull back}\emph{)} of the associtated sequence \emph{(}\ref{split Lambda}\emph{)} of groups is split, i.e., if for each $i=1,\ldots,k$ there is a group homomorphism $\sigma_i:\widehat{\Lambda}_i\rightarrow A^{\times}_{h,i}$ satisfying $q\circ\sigma_i=\id_{\widehat{\Lambda}_i}$. 
\end{proposition}

\begin{proof}
\,\,\,($``\Rightarrow"$) For this direction we first choose for each $i=1,\ldots, k$ an invertible element $a_i\in A_{\widehat{\lambda}_i}$ satisfying $a_i^{n_i}=1_A$. Then a short calculation shows that the map
\[\sigma_i:\widehat{\Lambda}_i\rightarrow A^{\times}_{h,i},\,\,\,\sigma(\widehat{\lambda}_i^l):=(a_i)^l,\,\,\,l=0,\ldots,n_i-1
\]defines a group homomorphism splitting the pull back extension $E_i$.

($``\Leftarrow"$) For each $i=1,\ldots, k$, let $\sigma_i:\widehat{\Lambda}_i\rightarrow A^{\times}_{h,i}$ be a group homomorphism splitting the the pull back extension $E_i$ and put $a_i:=\sigma(\widehat{\lambda}_i)$. Then $a_i\in A_{\widehat{\lambda}_i}$ is invertible by definition and satisfies $$a_i^{n_i}=\sigma(\widehat{\lambda}_i)^{n_1}=\sigma(\widehat{\lambda}_i^{n_i})=
\sigma(1)=1_A.$$
\end{proof}

\begin{lemma}\label{C_nC_mIV}
If $A$ is a unital locally convex algebra and $(A,\Lambda,\alpha)$ a dynamical system such that each isotypic component contains an invertible element, then the induced map
\[\sigma:\Gamma_A\times\Lambda\rightarrow\Gamma_A,\,\,\,\chi.\lambda:=\sigma(\chi,\lambda):=\chi\circ\alpha(\lambda)
\]defines a free action of $\Lambda$ on the spectrum $\Gamma_A$ of $A$. 
\end{lemma}

\begin{proof}
\,\,\,The proof of this assertion is similar to the proof of Lemma \ref{TNCPC_nB5}. Alternatively, we refer to \cite[Proposition 5.6]{Wa11c}.
\end{proof}

\begin{proposition}\label{C_nC_mV}
Let $P$ be a manifold and $(C^{\infty}(P),\Lambda,\alpha)$ a dynamical system such that each isotypic component contains an invertible element. Then the map 
\[\sigma:P\times\Lambda\rightarrow P,\,\,\,(\delta_p,\lambda)\mapsto\delta_p\circ\alpha(\lambda),
\]where we have identified $P$ with the set of characters via the map $\Phi$ from Remark \ref{TNCPC_nB6}, is smooth and defines a free and proper action of $\Lambda$ on the manifold $P$. In particular, we obtain a principal bundle $(P,P/\Lambda,\Lambda,\pr,\sigma)$, where $\pr:P\rightarrow P/\Lambda$ denotes the canonical orbit map.
\end{proposition}

\begin{proof}
\,\,\,Since the group $\Lambda$ is finite, the map $\sigma$ is obviously smooth and proper. The freeness of $\sigma$ is a consequence of Lemma \ref{C_nC_mIV}. Therefore, the Quotient Theorem implies that we obtain a principal bundle $(P,P/\Lambda,\Lambda,\pr,\sigma)$ (cf. \cite[Kapitel VIII, Satz 21.6]{To00}).
\end{proof}

\begin{theorem}\label{C_nC_mVI}\emph{(}Trivial principal $\Lambda$-bundles\emph{)}.
If $P$ is a manifold, then the following assertions hold:
\begin{itemize}
\item[\emph{(a)}]
If $(C^{\infty}(P),\Lambda,\alpha)$ is a smooth trivial NCP $\Lambda$-bundle, then the corresponding principal bundle $(P,P/\Lambda,\Lambda,\pr,\sigma)$ of Proposition \ref{C_nC_mV} is trivial.
\item[\emph{(b)}]
Conversely, if $(P,M,\Lambda,q,\sigma)$ is a trivial principal $\Lambda$-bundle, then the corresponding smooth dynamical system $(C^{\infty}(P),\Lambda,\alpha)$ 
defined by
\[\alpha:\Lambda\times C^{\infty}(P)\rightarrow C^{\infty}(P),\,\,\,\alpha(\lambda,f)(p):=(\lambda.f)(p):=f(\sigma(p,\lambda)),
\]is a trivial noncommutative principal $\Lambda$-bundle.
\end{itemize}
\end{theorem}

\begin{proof}
\,\,\,(a) If $(C^{\infty}(P),\Lambda,\alpha)$ is a trivial noncommutative principal $\Lambda$-bundle, then there exist a set of generators $\{\lambda_1,\ldots,\lambda_k\}$ of $\Lambda$ (having order $n_1,\ldots, n_k\in \mathbb{N}$) of minimal cardinality and invertible elements $f_i\in C^{\infty}(P)_{\widehat{\lambda}_i}$ satisfying $f_i^{n_i}=1$. From this we immediately conclude that $\im(f_i)=C_{n_i}$. Moreover, we note that the map
\[\mu:C_{n_1}\times\cdots\times C_{n_k}\rightarrow\Lambda,\,\,\,(\zeta_1^{l_1},\ldots\zeta_k^{l_k})\mapsto \lambda_1^{l_1}\cdots\lambda_1^{l_k},
\]is an isomorphism of finite abelian groups. Now, a short observation shows that
\[\varphi:=\pr\times(\mu\circ(f_1,\ldots,f_k)):P\rightarrow P/\Lambda\times\Lambda
\]defines an equivalence of principal $\Lambda$-bundles over $P/\Lambda$ implying that the principal bundle $(P,P/\Lambda,\Lambda,\pr,\sigma)$ of Proposition \ref{C_nC_mV} is trivial.

(b) Conversely, let $(P,M,\Lambda,q,\sigma)$ be a trivial principal $\Lambda$-bundle and let $\varphi:P\rightarrow M\times \Lambda$ be an equivalence of principal $\Lambda$-bundles over $M$. Then $\varphi$ induces a $\Lambda$-equivariant isomorphism of unital locally convex algebras between $C^{\infty}(P)$ and $C^{\infty}(M\times\Lambda)$ and thus the statment follows from the observation that 
\[C^{\infty}(M\times\Lambda)=\bigoplus_{\widehat{\lambda}\in\widehat{\Lambda}}
\widehat{\lambda}\cdot C^{\infty}(M).
\]holds as a consequence of Fourier decomposition.
\end{proof}

\section*{Some more examples} 

In this subsection we present a bunch of examples of trivial noncommutative principal $\Lambda$-bundles. For this we recall that the dual group $\widehat{\Lambda}$ also carries the structure of a finite abelian group. Moreover, given $n\in\mathbb{N}$, we will see that the matrix algebra $\M_n(\mathbb{C})$ carries the structure of a trivial noncommutative principal $C_n\times C_n$-bundle. 

\begin{construction}\label{l^1 spaces associated to dynamical systems again I}($\ell^1$-crossed products).
Let $(A,\Vert\cdot\Vert,^{*})$ be an involutive Banach algebra and $(A,\Lambda,\alpha)$ a dynamical system. Note that this means that $\Lambda$ acts by isometries of $A$. We write $F(\Lambda,A)$ for the vector space of functions $f:\Lambda\rightarrow A$ and define a multiplication on this space by
\[(f\star g)(\lambda):=\sum_{\lambda'\in\Lambda}f(\lambda')\alpha(\lambda',g(\lambda-\lambda')).
\]Moreover, an involution is given by
\[f^{*}(\lambda):=\alpha(\lambda,(f(-\lambda))^{*}).
\]These two operations are continuous for the $\ell^1$-norm 
\[\Vert f\Vert_1:=\sum_{\lambda\in\Lambda}\Vert f(\lambda)\Vert.
\]For consistency we write $\ell^1(A\rtimes_{\alpha} \Lambda)$ for the corresponding involutive Banach algebra.
\end{construction}

\begin{proposition}\label{continuous action again}
If $(A,\Vert\cdot\Vert,^{*})$ is an involutive Banach algebra and $(A,\widehat{\Lambda},\alpha)$ a dynamical system, then the map
\[\widehat{\alpha}:\Lambda\times\ell^1(A\rtimes_{\alpha} \widehat{\Lambda})\rightarrow\ell^1(A\rtimes_{\alpha}\widehat{\Lambda}),\,\,\,(\widehat{\alpha}(\lambda,f))(\widehat{\lambda}):=(\lambda.f)(\widehat{\lambda}):=\widehat{\lambda}(\lambda)\cdot f(\widehat{\lambda})
\]defines an action of $\Lambda$ on $\ell^1(A\rtimes_{\alpha}\widehat{\Lambda})$ by algebra automorphisms. In particular, the triple \[(\ell^1(A\rtimes_{\alpha} \widehat{\Lambda}),\Lambda,\widehat{\alpha})
\]defines a dynamical system.
\end{proposition}

\begin{proof}
\,\,\,The proof of this lemma is similar to the proof of \cite[Lemma 2.6]{Wa11a}.
\end{proof}

\begin{example}\label{l^1 as example again}
If $(A,\Vert\cdot\Vert,^{*})$ is an involutive Banach algebra and $(A,\widehat{\Lambda},\alpha)$ a dynamical system, then the dynamical system $(\ell^1(A\rtimes_{\alpha}\widehat{\Lambda}),\Lambda,\widehat{\alpha})$ is a trivial NCP $\Lambda$-bundle. Indeed, given a set of generators $\{\lambda_1,\ldots,\lambda_k\}$ of $\Lambda$ (having order $n_1,\ldots, n_k\in \mathbb{N}$) of minimal cardinality, we define
\[\delta_i(\widehat{\lambda}):=
\begin{cases}
1_A &\text{for}\,\,\,\widehat{\lambda}=\widehat{\lambda}_i\\
0 &\text{otherwise}
\end{cases}.
\]Then 
\[\lambda.\delta_i=\widehat{\lambda}_i(\lambda)\cdot\delta_i,\,\,\text{and}\,\,\,\delta_i\star\delta_i^{*}=\delta_i^{*}\star\delta_i={\bf 1}
\]show that $\delta_i$ is an invertible element of $\ell^1(A\rtimes_{\alpha}\widehat{\Lambda})$ lying in the isotypic component $\ell^1(A\rtimes_{\alpha}\widehat{\Lambda})_{\widehat{\lambda}_i}$. From this observation and $\delta_i^{n_i}={\bf 1}$ we conclude that $\ell^1(A\rtimes_{\alpha}\widehat{\Lambda})$ is a trivial noncommutative principal $\Lambda$-bundle.
\end{example}


\begin{example}\label{crossed product as example again}
If $(A,\Vert\cdot\Vert,^{*})$ is an involutive Banach algebra and $(A,\widehat{\Lambda},\alpha)$ a dynamical system, then the action $\widehat{\alpha}$ of Proposition \ref{continuous action again} extends to a continuous action of $\Lambda$ on the enveloping $C^*$-algebra $C^*(A\rtimes_{\alpha}\widehat{\Lambda})$ by algebra automorphisms. For details we refer to the \cite{Ta74}. In particular, the corresponding dynamical system 
\[(C^*(A\rtimes_{\alpha}\widehat{\Lambda}),\Lambda,\widehat{\alpha})
\]is a trivial noncommutative $\Lambda$-bundle. This follows exactly as in Example \ref{l^1 as example again}.
\end{example}

\begin{remark}(2-cocycles).\label{2-cocycles}
A 2-cocycle on $\Lambda$ with values in the circle $\mathbb{T}$ is a map $\omega:\Lambda\times\Lambda\rightarrow\mathbb{T}$ satisfying $\omega(0_{\Lambda},0_{\Lambda})=1$ and
\[\omega(\lambda,\lambda')\omega(\lambda+\lambda',\lambda'')=\omega(\lambda,\lambda'+\lambda'')\omega(\lambda',\lambda'')
\]for all $\lambda,\lambda',\lambda''\in\Lambda$. We write $Z^2(\Lambda,\mathbb{T})$ for the space of all 2-cocycle on $\Lambda$ with values in $\mathbb{T}$.
\end{remark}

\begin{construction}\label{l^1 spaces associated to cocycles I}($\ell^1$-spaces associated to $2$-cocycles).
Let $(A,\Vert\cdot\Vert,^{*})$ be an involutive Banach algebra and $\omega$ a 2-cocycle in $Z^2(\Lambda,\mathbb{T})$. The involutive Banach algebra $\ell^1(A\times_{\omega}\Lambda)$ is defined similar to Construction \ref{l^1 spaces associated to dynamical systems again I} by introducing a twisted multiplication 
\[(f\star g)(\lambda):=\sum_{\lambda'\in\lambda}f(\lambda')g(\lambda-\lambda')\omega(\lambda',\lambda-\lambda')
\]and an involution
\[(f^*)(\lambda):=\overline{\omega(\lambda,-\lambda)}\cdot f(-\lambda).
\]In fact, the cocycle property ensures that the multiplication is associative. The enveloping $C^*$-algebra of $\ell^1(A\times_{\omega} \Lambda)$ is called the \emph{twisted $A$-valued group $C^*$-algebra of $\Lambda$ by $\omega$} and denoted by $C^*(A\times_{\omega} \Lambda)$.
\end{construction}

\begin{example}\label{l^1 twisted as example again}
Let $(A,\Vert\cdot\Vert,^{*})$ be an involutive Banach algebra and $\omega$ a 2-cocycle in $Z^2(\widehat{\Lambda},\mathbb{T})$. Similarly to Proposition \ref{continuous action again}, we see that the map
\[\widehat{\alpha}:\Lambda\times\ell^1(A\times_{\omega}\widehat{\Lambda})\rightarrow\ell^1(A\times_{\omega}\widehat{\Lambda}),\,\,\,(\widehat{\alpha}(\lambda,f))(\widehat{\lambda}):=(\lambda.f)(\widehat{\lambda}):=\widehat{\lambda}(\lambda)\cdot f(\widehat{\lambda})
\]defines an action of $\Lambda$ on $\ell^1(A\times_{\omega} \Lambda)$ by algebra automorphisms. Moreover, the corresponding dynamical system
\[(\ell^1(A\times_{\omega}\widehat{\Lambda})),\Lambda,\widehat{\alpha})
\]turns out to be a trivial noncommutative principal $\Lambda$-bundle (cf. Example \ref{l^1 as example again}).
\end{example}

\begin{example}\label{c-stern crossed product again}
The action $\widehat{\alpha}$ of Example \ref{l^1 twisted as example again} extends to an action of $\Lambda$ on $C^*(A\times_{\omega}\widehat{\Lambda})$ by algebra automorphisms (cf. Example \ref{crossed product as example again}). The corresponding dynamical system
\[(C^*(A\times_{\omega}\widehat{\Lambda}),\Lambda,\widehat{\alpha})
\]is a trivial noncommutative principal $\Lambda$-bundle as well.
\end{example}

\begin{example}(The matrix algebra).\label{the matrix algebra}
For $n\in\mathbb{N}$ and $\zeta:=\exp(\frac{2\pi i}{n})$ we define
$$
R:=\begin{pmatrix}
 1      &  &  &  &    \\
         & \zeta &  &  &  \\
         &  & \zeta^2  &  &  \\
         &  &  & \ddots &  \\
         &  &  &  &  \zeta^{n-1} 
\end{pmatrix}
\,\,\,\text{and}\,\,\,
S:=\begin{pmatrix}
 0      & \ldots & \ldots & 0 & 1   \\
 1      & 0 & \ldots & 0 &  0\\
         & \ddots & \ddots & \vdots & \vdots \\
         &  & \ddots& 0 & \vdots \\
 0      &  &  & 1 & 0 
\end{pmatrix}.
$$

\noindent
We note that the matrices $R$ and $S$ are unitary and satisfy the relations $R^n={\bf 1}$, $S^n={\bf 1}$ and $RS=\zeta\cdot SR$. Further, they generate a $C^*$-subalgebra which clearly commutes only with the multiples of the identity, so it has to be the full matrix algebra. The map
\[\alpha:(C_n\times C_n)\times\M_n(\mathbb{C})\rightarrow\M_n(\mathbb{C}),\,\,\, (\zeta^k,\zeta^l).A:=R^lS^kAS^{-k}R^{-l}
\]for $k,l=0,1,\ldots,n-1$ defines a smooth group action of $C_n\times C_n$ on $\M_n(\mathbb{C})$ with fixed point algebra $\mathbb{C}$. We now conclude from
\[R^*\in\M_n(\mathbb{C})_{(1,0)}:=\left\{A\in\M_n(\mathbb{C}):\,(\forall (\zeta^k,\zeta^l)\in C_n\times C_n)\,(\zeta^k,\zeta^l).A=\zeta^k\cdot A\right\}
\]and
\[S\in\M_n(\mathbb{C})_{(0,1)}:=\left\{A\in\M_n(\mathbb{C}):\,(\forall (\zeta^k,\zeta^l)\in C_n\times C_n)\,(\zeta^k,\zeta^l).A=\zeta^l\cdot A\right\}
\]that the triple $(\M_n(\mathbb{C}),C_n\times C_n,\alpha)$ is a smooth trivial noncommutative principal $C_n\times C_n$-bundle.
\end{example}

\section*{Some remarks on the classification of trivial noncommutative principal bundles with finite abelian structure group}

In Section \ref{classoftrivnonC_n} we gave a complete classification of all ``algebraically" trivial noncommutative principal $C_n$-bundles with a prescribed fixed point algebra $B$. Unfortunately, a similar classification theory for a general noncommutative principal $\Lambda$-bundles seems to be much more involved. As before, Fourier decomposition shows that the underlying algebraic structure of a trivial noncommutative principal $\Lambda$-bundle $(A,\Lambda,\alpha)$ is the one of a so-called $(\widehat{\Lambda},A^{\Lambda})$-crossed product algebra, which are described and classified in Appendix \ref{classcroprodalg}. Thus, given a unital algebra $B$ (which serves as a ``base"), the main goal of this short subsection is to discuss some ideas and problems concerning a complete classification of $(\Lambda,B)$-crossed product algebras for which the corresponding pull back extensions of the associtated sequence (\ref{split Lambda}) of groups is split, i.e., to classify all `` algebraically" trivial noncommutative principal $\Lambda$-bundles with fixed point algebra $B$ (cf. Remark \ref{C_nC_mII} and Proposition \ref{C_nC_mIII}). 

\begin{notation}\label{shit I}
Let $B$ be a unital algebra. Then $\Ext(\Lambda,B)$ denotes the set of all equivalence classes of $(\Lambda,B)$-crossed product algebras (cf. Definition \ref{ext shit}) and we recall that according to Theorem \ref{classification on TNCT^B}, the map
\[\chi:\Ext(\Lambda,B)\rightarrow H^2(\Lambda,B),\,\,\,[A]\mapsto\chi(A)=[(S,\omega)]
\]is a well-defined bijection, i.e., the cohomology space $H^2(\Lambda,B)$ from Subsection \ref{factor systems} classifies set of all equivalence classes of $(\Lambda,B)$-crossed product algebras. Next, given a set $\{\lambda_1,\ldots,\lambda_k\}$ of generators of $\Lambda$ (having order $n_1,\ldots, n_k\in \mathbb{N}$) of minimal cardinality and a class $[A]\in\Ext(\Lambda,B)$ with $\chi(A)=[(S,\omega)]$, we write $\Lambda_i$ for the subgroup of $\Lambda$ generated by $\lambda_i$, i.e., $\Lambda_i:=\langle\lambda_i\rangle_{\text{gr}}$ and $[(S_i,\omega_i)]\in H^2(\Lambda_i,B)$ for the factor system of the corresponding pull back extension $E_i$ (cf. Remark \ref{C_nC_mII}). 
\end{notation}

\begin{definition}\label{shit II}
Let $B$ be a unital algebra. We write $\Ext(\Lambda,B)_{\text{ss}}$ for the set of all equivalence classes of $(\Lambda,B)$-crossed product algebras for which there exists a set $\{\lambda_1,\ldots,\lambda_k\}$ of generators of $\Lambda$ (having order $n_1,\ldots, n_k\in \mathbb{N}$) of minimal cardinality such that $[A_{(S_i,\omega_i)}]\in \Ext(\Lambda_i,B)_{\text{split}}$ for all $i=1,\ldots,k$ (cf. Definition \ref{ClassTriNonCommPB3} and Construction \ref{realization of TNCT^B from factor systems I}). Here, ``$\text{ss}$" stands for ``semisplit". 
\end{definition}

\begin{example}\label{shit IV}
If $B=\mathbb{C}$, then 
we conclude from Remark \ref{B abelian} that $\Ext(\Lambda,\mathbb{C})$ is classified by the space $H^2(\Lambda,\mathbb{C}^{\times})$, which in turn can easily be computed with the help of \cite[Proposition II.4]{Ne07b}. For example, if $\Lambda=C_n\times C_n$, then 
\[H^2(C_n\times C_n,\mathbb{C}^{\times})\cong C_n.
\]A short observation shows that the class of the trivial noncommutative principal $C_n\times C_n$-bundle $(\M_n(\mathbb{C}),C_n\times C_n,\alpha)$ of Example \ref{the matrix algebra} corresponds to the element $\exp(\frac{2\pi i}{n})$. 
\end{example}


We close this short subsection with the following open problem:

\begin{open problem}\label{shit VI}
Find a computable description of the set $\Ext(\Lambda,B)_{\text{ss}}$. For this purpose, it is worth to study \cite[Corollary 2.21]{Ne07a}. Moreover, ideas of the classification methods for loop algebras of \cite{ABP04} might be useful. The algebras $\M_m(\mathbb{C})$ and $C^{\infty}(M)$ for $m\in\mathbb{N}$ and a manifold $M$ are of particular interest and serve also as a good starting point. We recall that $\Ext(\Lambda,\M_m(\mathbb{C}))$ is classified by $H^2(\Lambda,\mathbb{C}^{\times})$ (cf. Theorem \ref{class of Z-kernels II}), which in turn can easily be computed with the help of \cite[Proposition II.4]{Ne07b}.
\end{open problem}

\section{Non-trivial noncommutative principal bundles with finite abelian structure sroup}\label{NonTriPrinBundFinAbStrGr}

The main goal of this section is to present a geometrically oriented approach to the noncommutative geometry of principal $\Lambda$-bundles. Since the freeness property of a group action is a local condition (cf. \cite[Remark 8.10]{Wa11d}), our main idea is inspired by the classical setting: Loosely speaking, a dynamical system $(A,\Lambda,\alpha)$ is called a noncommutative principal $\Lambda$-bundle, if it is ``locally" a trivial noncommutative principal $\Lambda$-bundle in the sense of Section \ref{TriPrinBundFinAbStrGr}. We prove that this approach extends the classical theory of principal $\Lambda$-bundles and present some noncommutative examples. In fact, we first show that each trivial noncommutative principal $\Lambda$-bundle carries the structure of a noncommutative principal $\Lambda$-bundle in its own right. We further show that examples of noncommutative principal $\Lambda$-bundles are provided by sections of algebra bundles with trivial noncommutative principal $\Lambda$-bundle as fibre, sections of algebra bundles which are pull-backs of principal $\Lambda$-bundles and sections of trivial equivariant algebra bundles. 

\begin{notation}
Let $A$ be a unital locally convex algebra and $C^{\infty}(\mathbb{R},A)$ the algebra of smooth $A$-valued function on $\mathbb{R}$ endowed with the smooth compact open topology. For each $a\in A$ we define an element of $C^{\infty}(\mathbb{R},A)$ by $f_a:\mathbb{R}\rightarrow A$, $t\mapsto 1_A-ta$ and write $I_a$ for the closure of the two-sided ideal generated by this element. We further write $A_{\{a\}}:=C^{\infty}(\mathbb{R},A)/I_a$ for the corresponding locally convex quotient algebra and call it the \emph{smooth localization} of $A$ with respect to $a$. Given a dynamical system $(A,\Lambda,\alpha)$ and an element $z$ in the fixed point algebra of the induced action of $\Lambda$ on the center $C_A$ of $A$, i.e., an element $z\in C_A^{\Lambda}$, the map
\[\alpha_{\{z\}}:\Lambda\times A_{\{z\}}\rightarrow A_{\{z\}},\,\,\,(\lambda,[f]\underline{})\mapsto[\alpha(\lambda)\circ f]
\]defines an action of $\Lambda$ on $A_{\{z\}}$ by algebra automorphisms. In particular, the triple $(A_{\{z\}},\Lambda,\alpha_{\{z\}})$ is a dynamical system and is called the \emph{smooth localization of $(A,\Lambda,\alpha)$} associated to the element $z\in C_A^{\Lambda}$. For a solid background on the previous discussion we refer to the first part of \cite{Wa11d}.
\end{notation}

\begin{definition}\label{NCPT^nB again}
(Noncommutative principal $\Lambda$-bundles). We call a dynamical system $(A,\Lambda,\alpha)$ a \emph{noncommutative principal $\Lambda$-bundle} if for each $\chi\in\Gamma_{C_A^{\Lambda}}$ there exists an element $z\in C_A^{\Lambda}$ with $\chi(z)\neq 0$ such that the corresponding localized dynamical system $(A_{\{z\}},\Lambda,\alpha_{\{z\}})$ is a trivial noncommutative principal $\Lambda$-bundle the sense of Section \ref{TriPrinBundFinAbStrGr}. 
\end{definition}

\begin{remark}\label{remark on NCPT^nB again}
(An equivalent point of view). The previous definition of noncommutative principal $\Lambda$-bundles is inspired by the classical setting. Indeed, given $z\in C_A^{\Lambda}$, we recall that $D(z):=\{\chi\in\Gamma_{C_A^{\Lambda}}:\,\chi(z)\neq 0\}$. Then a short observation shows that a dynamical system $(A,\Lambda,\alpha)$ is a noncommutative principal $\Lambda$-bundle if and only if there exists a family of elements $(z_i)_{i\in I}\subseteq C_A^{\Lambda}$ satisfying the following two conditions:
\begin{itemize}
\item[(i)]
The family $(D(z_i))_{i\in I}$ is an open covering of $\Gamma_{C_A^{\Lambda}}$.
\item[(ii)]
The localized dynamical systems $(A_{\{z_i\}},\Lambda,\alpha_{\{z_i\}})$ are trivial noncommutative principal $\Lambda$-bundles.
\end{itemize}
\end{remark}

\begin{theorem}\label{NCT^nB for manifold again}\emph{(}Reconstruction Theorem\emph{)}.
For a manifold $P$, the following assertions hold:
\begin{itemize}
\item[\emph{(a)}]
If $P$ is compact and $(C^{\infty}(P),\Lambda,\alpha)$ is a noncommutative principal $\Lambda$-bundle, then we obtain a principal $\Lambda$-bundle $(P,P/\Lambda,\Lambda,\pr,\sigma)$.
\item[\emph{(b)}]
Conversely, if $(P,M,\Lambda,q,\sigma)$ is a principal $\Lambda$-bundle, then the corresponding dynamical system $(C^{\infty}(P),\Lambda,\alpha)$ is a noncommutative principal $\Lambda$-bundle.
\end{itemize}
\end{theorem}

\begin{proof}
\,\,\,The proof of this theorem is similar to the proof of \cite[Theorem 8.11]{Wa11d}. In fact, for part (a) we recall that if $A$ is a commutative unital locally convex algebra and $(A,\Lambda,\alpha)$ a trivial noncommutative principal $\Lambda$-bundle, then the induced action of $\Lambda$ on the spectrum $\Gamma_A$ of $A$ is free (cf. Propositon \ref{C_nC_mV}). For part (b) we have to add Theorem \ref{C_nC_mVI}.
\end{proof}

\section*{Example 1: Trivial noncommutative principal $\Lambda$-bundles}

We show that each trivial noncommutative principal $\Lambda$-bundle carries the structure of a noncommutative principal $\Lambda$-bundle:

\begin{theorem}\emph{(}Trivial Noncommutative Principal $\Lambda$-bundles\emph{)}.\label{TNCTB are NCTB}
Each trivial noncommutative principal $\Lambda$-bundle $(A,\Lambda,\alpha)$ carries the structure of a noncommutative principal $\Lambda$-bundle.
\end{theorem}

\begin{proof}
\,\,\,If $(A,\Lambda,\alpha)$ is a trivial noncommutative principal $\Lambda$-bundle, then \mbox{$1_A\in C_A^{\Lambda}$} and $\chi(1_A)=1\neq 0$ holds for each $\chi\in\Gamma_{C_A^{\Lambda}}$. In particular, \cite[Corollary 2.10]{Wa11d} implies that $A_{\{1_A\}}\cong A$ holds as unital locally convex algebras. Thus, $(A,\Lambda,\alpha)$ is a noncommutative principal $\Lambda$-bundle in its own right.
\end{proof}

\section*{Example 2: Sections of algebra bundles with a trivial noncommutative principal $\Lambda$-bundle as fibre}\label{examples of NCP T^n-bundles II}

We show that if $A$ is a unital Fr\'echet algebra and $(A,\Lambda,\alpha)$ a trivial noncommutative principal $\Lambda$-bundle such that $C_A^{\Lambda}$ is isomorphic to $\mathbb{C}$, then the algebra of sections of each algebra bundle with ``fibre" $(A,\Lambda,\alpha)$ is a noncommutative principal $\Lambda$-bundle. We start with the following lemma:

\begin{proposition}\label{trivial NCP T^n-bundles from trivial algebra bundles}
If $(A,\Lambda,\alpha)$ is a trivial noncommutative principal $\Lambda$-bundle and $M$ a manifold, then the triple $(C^{\infty}(M,A),\Lambda,\beta)$, where
\[\beta:\Lambda\times C^{\infty}(M,A)\rightarrow C^{\infty}(M,A),\,\,\,(\lambda,f)\mapsto\alpha(\lambda)\circ f
\]carries the structure of a trivial noncommutative principal $\Lambda$-bundle.
\end{proposition}

\begin{proof}
\,\,\,The claim directly follows from \cite[Lemma 7.5]{Wa11d} and the fact that the algebra $A$ is naturally embedded in $C^{\infty}(M,A)$ through the constant maps. In fact, if $\widehat{\lambda}\in\widehat{\Lambda}$ (again, $\widehat{\Lambda}$ denotes the dual group of $\Lambda$), then the corresponding isotypic component $C^{\infty}(M,A)_{\widehat{\lambda}}$ is equal to $C^{\infty}(M,A_{\widehat{\lambda}})$. 
\end{proof}

\begin{remark}\label{automorphism of dynamical systems}(The automorphism group of a dynamical system).
Let $(A,\Lambda,\alpha)$ be a dynamical system. The group
\[\Aut_{\Lambda}(A):=\{\varphi\in\Aut(A):\,(\forall \lambda\in\Lambda)\,\alpha(\lambda)\circ\varphi=\varphi\circ\alpha(\lambda)\}
\]is called the \emph{automorphism group of} the dynamical system $(A,\Lambda,\alpha)$.
\end{remark}

\begin{example}\label{G-automorphism group of quantumtori}
Let $(\M_n(\mathbb{C}),C_n\times C_n,\alpha)$ be the trivial noncommutative principal $C_n\times C_n$-bundle of Example \ref{the matrix algebra}. Then we have
\[\Aut_{C_n\times C_n}(\M_n(\mathbb{C}))\cong C_n\times C_n.
\]Indeed, we first recall that $\M_n(\mathbb{C})$ is generated by the unitaries $R$ and $S$. If now $\varphi\in\Aut_{C_n\times C_n}(\M_n(\mathbb{C}))$, then a short observation leads to $\varphi(R)=\lambda_R\cdot R$ and to $\varphi(S)=\lambda_S\cdot S$ for some $\lambda_R,\lambda_S\in C_n$. In particular, each element in $\Aut_{C_n\times C_n}(\M_n(\mathbb{C}))$ corresponds to an element in $C_n\times C_n$ and vice versa. 
\end{example}

\begin{theorem}\label{NCP G_bundle with fibre trivial NCP G bundle}
Let $A$ be a unital Fr\'echet algebra and $(A,\Lambda,\alpha)$ a trivial noncommutative principal $\Lambda$-bundle such that $C_A^{\Lambda}$ is isomorphic to $\mathbb{C}$. Further, let $M$ be a manifold, $(U_i)_{i \in I}$ an open cover of $M$ and $U_{ij}:=U_i\cap U_j$ for $i,j\in I$. If $(g_{ij})_{i,j \in I}$ is a collection of functions $g_{ij}\in C^{\infty}(U_{ij},\Aut_{\Lambda}(A))$ satisfying 
\[g_{ii}={\bf 1}\,\,\,\text{and}\,\,\,g_{ij}g_{jk}=g_{ik}\,\,\,\text{on}\,\,\,U_{ijk}:=U_i\cap U_j\cap U_k,
\]then the following assertions hold:
\begin{itemize}
\item[\emph{(a)}]
There exists an algebra bundle $(\mathbb{A},M,A,q)$ and bundle charts $\varphi_i:U_i\times A\rightarrow\mathbb{A}_{U_i}$ such that
\[(\varphi_i^{-1}\circ\varphi_j)(x,a)=(x,g_{ij}(x).a).
\]Moreover, the map
\[\sigma:\Lambda\times\mathbb{A}\rightarrow\mathbb{A},\,\,\,(\lambda,a)\mapsto\varphi_i(x,\alpha(\lambda).a_0),
\]where $i\in I$ with $x=q(a)\in U_i$ and $a_0\in A$ with $\varphi_i(x,a_0)=a$, defines an action of $\Lambda$ on $\mathbb{A}$ by fibrewise algebra automorphisms.
\item[\emph{(b)}]
The map
\[\beta:\Lambda\times\Gamma\mathbb{A}\rightarrow\Gamma\mathbb{A},\,\,\,\beta(\lambda,s)(m):=\sigma(\lambda,s(m))
\]defines an action of $\Lambda$ on the corresponding space $\Gamma\mathbb{A}$ of sections by algebra automorphisms. Furthermore, the triple $(\Gamma\mathbb{A},\Lambda,\beta)$ carries the structure of a noncommutative principal $\Lambda$-bundle.
\end{itemize}
\end{theorem}

\begin{proof}
\,\,\,The first assertion of this theorem is a consequence of \cite[Proposition 8.17]{Wa11d}. The proof of the second assertion is similar to the proof of \cite[Theorem 8.20]{Wa11d}. In fact, we just have to add Proposition \ref{trivial NCP T^n-bundles from trivial algebra bundles}.
\end{proof}

\begin{example}\label{non-triviality of the previous construction for example 2}(Non-triviality of the previous construction).
In this example we show that the previous construction actually leads to non-trivial examples. For this we apply Theorem \ref{NCP G_bundle with fibre trivial NCP G bundle} to the trivial noncommutative principal $C_n\times C_n$-bundle $(\M_n(\mathbb{C}),C_n\times C_n,\alpha)$ of Example \ref{the matrix algebra}. In view of Example \ref{G-automorphism group of quantumtori} we have 
\begin{align}
\Aut_{C_n\times C_n}(\M_n(\mathbb{C}))\cong C_n\times C_n.\notag
\end{align}
In particular, a similar argument as in \cite[Proposition 2.1.14]{Wa11a} implies that there is a one-to-one correspondence between the algebra bundles arising from Theorem \ref{NCP G_bundle with fibre trivial NCP G bundle} (a) and principal $C_n\times C_n$-bundles. Thus, if $(\mathbb{A},M,\M_n(\mathbb{C}),q)$ is such an algebra bundle which corresponds to a non-trivial principal $C_n\times C_n$-bundle, then also $(\mathbb{A},M,\M_n(\mathbb{C}),q)$ is non-trivial as algebra bundle. We claim that the associated dynamical system $(\Gamma\mathbb{A},C_n\times C_n,\beta)$ of Theorem \ref{NCP G_bundle with fibre trivial NCP G bundle} (b) is a non-trivial noncommutative principal $C_n\times C_n$-bundle. To prove this claim we assume the converse, i.e., that $(\Gamma\mathbb{A},C_n\times C_n,\beta)$ is a trivial noncommutative principal $C_n\times C_n$-bundle. For this, we proceed as follows:


(i) Since $(\Gamma\mathbb{A},C_n\times C_n,\beta)$ is assumed to be a trivial noncommutative principal $C_n\times C_n$-bundle, there exist two generators $\lambda_1$, $\lambda_2$ of $C_n\times C_n$ and invertible elements 
\[s_1\in(\Gamma\mathbb{A})_{\widehat{\lambda}_1}=\left\{s\in\Gamma\mathbb{A}:\,(\forall m\in M)\,s(m)\in(\mathbb{A}_m)_{\widehat{\lambda}_1}\right\}
\]and
\[s_2\in(\Gamma\mathbb{A})_{\widehat{\lambda}_2}=\left\{s\in\Gamma\mathbb{A}:\,(\forall m\in M)\,s(m)\in(\mathbb{A}_m)_{\widehat{\lambda}_2}\right\}
\]satisfying $s_1^n=s_2^n=1$. 

(ii) If $s^i_1:=\pr_{\M_n(\mathbb{C})}\circ\varphi_i^{-1}\circ {s_1}_{\mid U_i}:U_i\rightarrow \M_n(\mathbb{C})$, then a short observation shows that $s^i_1=\mu^i_1\cdot V_1$ for a unitary $V_1\in\M_n(\mathbb{C})_{\widehat{\lambda}_1}$ satisfying $V_1^n={\bf 1}$ (note that $V_1$ has the form $R^kS^l$ for some $k,l\in\mathbb{N}$) and a (smooth) function $\mu^i_1:U_i\rightarrow C_n$. The same construction applied to $s_2$ shows that $s^i_2=\mu^i_2\cdot V_2$ for a unitary $V_2\in\M_n(\mathbb{C})_{\widehat{\lambda}_2}$ satisfying $V_2^n={\bf 1}$ and a (smooth) function $\mu^i_2:U_i\rightarrow C_n$.

(iii) Next, we recall that $\M_n(\mathbb{C})$ is generated by the unitaries $R$ and $S$. Therefore it is also generated by the unitaries $V_1$ and $V_2$, i.e., its elements can uniquely be written as finite sums of the form
\[a=\sum a_{kl}V_1^kV_2^l\,\,\,\text{with}\,\,\,a_{kl}\in\mathbb{C}.
\]

(iv) We now show that the map
\[\varphi:M\times\M_n(\mathbb{C})\rightarrow\mathbb{A},\,\,\,\left(m,a=\sum a_{kl}V_1^kV_2^l\right)\mapsto\sum a_{kl}s_1^k(m)\cdot s_2^l(m)
\]is an equivalence of algebra bundles over $M$. Indeed, $\varphi$ is bijective and fibrewise an algebra automorphism. Moreover, the map $\varphi$ is smooth if and only if the map 
\[\psi_i:=\pr_{\M_n(\mathbb{C})}\circ\varphi_i^{-1}\circ\varphi_{\mid U_i\times\M_n(\mathbb{C})}:U_i\times \M_n(\mathbb{C})\rightarrow\M_n(\mathbb{C}), 
\]
\[\left(x,a=\sum a_{kl}V_1^kV_2^l\right)\mapsto\sum a_{kl}(s^i_1)^k(x)\cdot(s^i_2)^l(x)
\]is smooth for each $i\in I$. Since 
\[\psi_i(x,a)=\sum a_{kl}(\mu^i_1)^k(x)\cdot(\mu^i_2)^l(x)V_1^kV_2^l,
\]we conclude that $$\psi_i=\alpha\circ\left((\mu^i_1,\mu^i_2)\times\id_{\M_n(\mathbb{C})}\right),$$i.e., that $\psi_i$ is smooth as a composition of smooth maps. A similar argument shows the smoothness of the inverse map.

(v) We finally achieve the desired contradiction: In view of part (iv), $\mathbb{A}$ is a trivial algebra bundle contradicting the construction of $\mathbb{A}$, i.e., that $\mathbb{A}$ is non-trivial as algebra bundle. This proves the claim.
\end{example}

\section*{Example 3: Sections of algebra bundles which are pull-backs of principal $\Lambda$-bundles}\label{examples of NCP T^n-bundles III}

We show that if $A$ is a unital Fr\'echet algebra with trivial center, $(\mathbb{A},M,A,q)$ an algebra bundle and $(P,M,\Lambda,\pi,\sigma)$ a principal $\Lambda$-bundle, then the algebra of sections of the pull-back bundle 
\[\pi^{*}(\mathbb{A}):=\{(p,a)\in P\times\mathbb{A}:\,\pi(p)=q(a)\}
\]carries the structure of a noncommutative principal $\Lambda$-bundle. We start with the following lemma:

\begin{proposition}\label{pull-back 0}
If $A$ is a unital locally convex algebra and $M$ a manifold, then the map
\[\alpha:\Lambda\times C^{\infty}(M\times \Lambda,A)\rightarrow C^{\infty}(M\times \Lambda,A),\,\,\,(\lambda,f)\mapsto(\lambda.f)(m,\lambda'):=f(m,\lambda\lambda')
\]defines an action of $\Lambda$ on $C^{\infty}(M\times\Lambda,A)$ by algebra automorphisms. In particular, the triple $(C^{\infty}(M\times\Lambda,A),\Lambda,\alpha)$ carries the structure of a trivial noncommutative principal $\Lambda$-bundle.
\end{proposition}

\begin{proof}
\,\,\,For the proof we just have to note that the algebra $C^{\infty}(M\times\Lambda)$ is naturally embedded in $C^{\infty}(M\times\Lambda,A)$ through the unit element of $A$. 
\end{proof}

\begin{theorem}\label{pull-back IV}
Let $A$ be a unital Fr\'echet algebra with trivial center, $(\mathbb{A},M,A,q)$ an algebra bundle and $(P,M,\Lambda,\pi,\sigma)$ a principal bundle. If $\pi^{*}(\mathbb{A})$ is the pull-back bundle over $P$ and $\mathcal{A}:=\Gamma\pi^{*}(\mathbb{A})$ the corresponding space of sections, then the following assertions hold:
\begin{itemize}
\item[\emph{(a)}]
The map
\[\sigma^*:\pi^{*}(\mathbb{A})\times \Lambda\rightarrow\pi^{*}(\mathbb{A}),\,\,\,((p,a),\lambda)\mapsto(p.\lambda,a)
\]defines an action of $\Lambda$ on $\pi^{*}(\mathbb{A})$.
Moreover, the map
\[\alpha:\Lambda\times\mathcal{A}\rightarrow\mathcal{A},\,\,\,\alpha(\lambda,s)(p):=\sigma^*(s(p.\lambda),\lambda^{-1})
\]defines an action of $\Lambda$ on $\mathcal{A}$ by algebra automorphisms.
\item[\emph{(b)}]
The dynamical system $(\mathcal{A},\Lambda,\alpha)$ carries the structure of a noncommutative principal $\Lambda$-bundle.
\end{itemize}
\end{theorem}

\begin{proof}
\,\,\,The first assertion of this theorem is a consequence of \cite[Lemma 8.24 and Proposition 8.25]{Wa11d}. The proof of the second assertion is similar to the proof of \cite[Theorem 8.26]{Wa11d}. In fact, we just have to add Proposition \ref{pull-back 0}.
\end{proof}

\begin{example}\label{non-triviality of the previous construction for example 3}(Non-triviality of the previous construction).
In this example we show that the previous construction actually leads to non-trivial examples. Therefore let $n\in\mathbb{N}$ with $n>1$ and $\zeta:=\exp(\frac{2\pi i}{n})$. Further, let $m\in\mathbb{N}$ such that $\zeta^m=\zeta$ (e.g. $m=n+1$) 
and choose any algebra bundle $\mathbb{A}$ over $\mathbb{T}^2$ with fibre $M_m(\mathbb{C})$ (e.g. take the smooth two-dimensional quantum torus $\mathbb{T}^2_{\frac{1}{m}}$ (cf. \cite[Remark 8.22]{Wa11d})). 
Then, the pull-back along the non-trivial principal bundle (covering) defined by the natural action of $C_n\times C_n$ on $\mathbb{T}^2$, i.e., by 
\[(t_1,t_2).(\zeta^k,\zeta^l):=(\zeta^k\cdot t_1,\zeta^l\cdot t_2)
\]for $k,l=0,1,\ldots,n-1$, leads to an algebra bundle $\pi^{*}(\mathbb{A})$ over $\mathbb{T}^2$ with fibre $M_m(\mathbb{C})$. 
We claim that the associated dynamical system $(\mathcal{A},C_n\times C_n,\alpha)$ of Theorem \ref{pull-back IV} (a) is a non-trivial noncommutative principal $C_n\times C_n$-bundle. To prove this claim we assume the converse, i.e., that $(\mathcal{A},C_n\times C_n,\alpha)$ is a trivial noncommutative principal $C_n\times C_n$-bundle and proceed as follows:

(i) In the following let $(\varphi_i,U_i)_{i\in I}$ be a bundle atlas of the pull-back bundle $\pi^{*}(\mathbb{A})$ over $\mathbb{T}^2$. For a section $s\in\mathcal{A}$ and $i\in I$ we write 
\[s_i:=\pr_{M_m(\mathbb{C})}\circ\varphi_i^{-1}\circ s_{\mid U_i}
\]for the corresponding function in $C^{\infty}(U_i,M_m(\mathbb{C}))$ (cf. \cite[Construction 4.3]{Wa11d}). We recall that $s_j(t)=g_{ji}(t).s_i(t)$ holds for all $i,j\in I$ and $t\in U_i\cap U_j$, where 
\[g_{ji}:(U_i\cap U_j)\rightarrow \Aut(M_m(\mathbb{C}))
\]denotes the smooth map defined by the transition function $\varphi_j^{-1}\circ\varphi_i$.

(ii) Let $s\in\mathcal{A}$. We show that the map
\[\Det(s):\mathbb{T}^2\rightarrow\mathbb{C},\,\,\,\Det(s)(t):=\det(s_i(t)),
\]for $i\in I$ with $z\in U_i$, is well-defined and smooth: In fact, the crucial point is to show that the map $\Det(s)$ is well-defined: For this let $i,j\in I$ with $z\in U_i\cap U_j$. Since each automorphism of the matrix algebra $\M_m(\mathbb{C})$ is inner (cf. the well-known Skolem--Noether Theorem), we easily conclude that
\[\det(s_j(t))=\det(g_{ji}(t)\cdot s_i(t))=\det(s_i(t)).
\]The smoothness of the map $\Det(s)$ follows from the local description by a smooth function.

(iii) Since $(\mathcal{A},C_n\times C_n,\alpha)$ is assumed to be a trivial noncommutative principal $C_n\times C_n$-bundle, there exist two generators $\lambda_1$, $\lambda_2$ of $C_n\times C_n$ and invertible elements $F_1\in\mathcal{A}_{\widehat{\lambda}_1}$ and $F_2\in\mathcal{A}_{\widehat{\lambda}_2}$ satisfying $F_1^n=F_2^n=1_{\mathcal{A}}$. 

(iv) Part (ii) now implies that $f_1:=\Det(F_1)\in C^{\infty}(\mathbb{T}^2)$. Since $F_1$ is invertible, so is $f_1$, i.e., $f_1$ takes values in $\mathbb{C}^{\times}$. Moreover, the function $f_1$ satisfies
\begin{align}
(\zeta^k,\zeta^l).f_1&=(\zeta^k,\zeta^l).\Det(F_1)=\Det(\widehat{\lambda}_1(\zeta^k,\zeta^l)\cdot F_1)\notag\\
&=(\widehat{\lambda}_1(\zeta^k,\zeta^l))^m\cdot\Det(F_1)=\widehat{\lambda}_1^m(\zeta^k,\zeta^l)\cdot\Det(F_1)\notag\\
&=\widehat{\lambda}_1(\zeta^k,\zeta^l)\cdot\Det(F_1)=\widehat{\lambda}_1(\zeta^k,\zeta^l)\cdot f_1\notag
\end{align}
for all $k,l=0,1,\ldots,n-1$. Thus, $f_1$ is an invertible element in $C^{\infty}(\mathbb{T}^2)_{\widehat{\lambda}_1}$ satisfying $f_1^n=1$ (here we have used that the action of $C_n\times C_n$ on $\mathcal{A}$ restricts to an action on $C_{\mathcal{A}}\cong C^{\infty}(\mathbb{T}^2)$).  The same construction applied to $F_2$ gives an invertible element $f_2\in C^{\infty}(\mathbb{T}^2)_{\widehat{\lambda}_2}$ satisfying $(f_2)^n=1$.

(v) Finally, part (iv) leads to a contradiction: Indeed, we conclude just as in the proof of Theorem \ref{C_nC_mVI} (a) that the smooth functions $f_1$ and $f_2$ have image $C_n$ and define an equivalence of principal $C_n\times C_n$-bundles over $\mathbb{T}^2/(C_n\times C_n)\cong\mathbb{T}^2$. But this is not possible, since $\mathbb{T}^2$ is connected.
\end{example}

\section*{Example 4: Sections of trivial equivariant algebra bundles}\label{examples of NCP T^n-bundles V}

We now consider again a principal bundle $(P,M,\Lambda,q,\sigma)$ and, in addition, a unital locally convex algebra $A$. If $\pi:\Lambda\times A\rightarrow A$ defines a smooth action of $\Lambda$ on $A$ by algebra automorphisms, then $(p,a).\lambda:=(p.\lambda,\pi(\lambda^{-1}).a)$ defines a (free) action of $\Lambda$ on $P\times A$ and one easily verifies that the trivial algebra bundle $(P\times A,P,A,q_P)$ is $\Lambda$-equivariant. Moreover, a short observation shows that the map
\[\alpha:\Lambda\times C^{\infty}(P,A)\rightarrow C^{\infty}(P,A),\,\,\,(\lambda.f)(p):=\pi(\lambda).f(p.\lambda)
\]defines a smooth action of $\Lambda$ on $C^{\infty}(P,A)$ by algebra automorphisms. We recall that the corresponding fixed point algebra 
is isomorphic (as $C^{\infty}(M)$-algebra) to the space of sections of the associated algebra bundle
\[\mathbb{A}:=P\times_{\pi}A:=P\times_{\Lambda}A:=(P\times A)/\Lambda
\]over $M$. In fact, bundle charts of $(P,M,\Lambda,q,\sigma)$ induces bundle charts for the associated algebra bundle.

\begin{lemma}\label{V.1}
If we apply the previous situation to the trivial principal bundle $(M\times\Lambda,M,\Lambda,q_M,\sigma_{\Lambda})$, then the corresponding dynamical system $(C^{\infty}(M\times\Lambda,A),\Lambda,\alpha)$ carries the structure of a trivial noncommutative principal $\Lambda$-bundle with fixed point algebra $C^{\infty}(M,A)$.
\end{lemma}

\begin{proof}
\,\,\,For the proof we again just have to note that the algebra $C^{\infty}(M\times\Lambda)$ is naturally embedded in $C^{\infty}(M\times\Lambda,A)$ through the unit element of $A$. 
\end{proof}

\begin{theorem}\label{V.2}
If $(P,M,\Lambda,\pi,\sigma)$ is a principal bundle, $A$ a unital Fr\'echet algebra with trivial center and $\pi:\Lambda\times A\rightarrow A$ a smooth action of $\Lambda$ on $A$, then the dynamical system $(C^{\infty}(P,A),\Lambda,\alpha)$ is a noncommutative principal $\Lambda$-bundle.
\end{theorem}

\begin{proof}
\,\,\,The proof of this theorem is similar to the proof of \cite[Theorem 8.30]{Wa11d}. In fact, we just have to add Lemma \ref{V.1}).
\end{proof}

\begin{example}\label{V.3}
We want to apply Theorem \ref{V.2} to the the non-trivial principal bundle (covering) defined by the natural action of $C_m\times C_m$ on $\mathbb{T}^2$, i.e., by 
\[(z_1,z_2).(\zeta^k,\zeta^l):=(\zeta^k\cdot z_1,\zeta^l\cdot z_2)
\]for $k,l=0,1,\ldots,m-1$, the algebra $M_m(\mathbb{C})$ and the action of $C_m\times C_m$ on $M_m(\mathbb{C})$ defined by 
\[(\zeta^k,\zeta^l).A:=R^lS^kAS^{-k}R^{-l}
\]for $k,l=0,1,\ldots,m-1$. Here, $R$ and $S$ are defined as in Example \ref{the matrix algebra} ($\theta=\frac{1}{m}$). The corresponding dynamical system $(C^{\infty}(\mathbb{T}^2,M_m(\mathbb{C})),C_m\times C_m,\alpha)$ is a noncommutative principal $C_m\times C_m$-bundle with fixed point algebra the rational quantum torus $\mathbb{T}^2_{\frac{1}{m}}$ (cf. \cite[Appendix E, Proposition E.2.5]{Wa11b}). According to Example \ref{the matrix algebra}, the algebra $M_m(\mathbb{C})$ carries the structure of a trivial noncommutative principal $C_m\times C_m$-bundle. Therefore, it turns out that the same holds for $(C^{\infty}(\mathbb{T}^2,M_m(\mathbb{C})),C_m\times C_m,\alpha)$ since $M_m(\mathbb{C})$ is naturally embedded in the algebra $C^{\infty}(\mathbb{T}^2,M_m(\mathbb{C}))$ through the constant maps.
\end{example}

\begin{remark}\label{non-triviality of the previous construction for example 4}(Non-triviality of the previous construction).
Non-trivial examples can be constructed similarly as in Example \ref{non-triviality of the previous construction for example 3}.
\end{remark}

\appendix
{\par
   \setcounter{section}{0}%
   \renewcommand\thesection{\Alph{section}}

\section{Classification of crossed product algebras}\label{classcroprodalg}

Let $G$ be a group and $B$ be a unital algebra. A $(G,B)$-crossed product algebra is a $G$-graded unital algebra with $A_{1_G}=B$ and the additional property that each grading space contains an invertible element. For example, if $G$ is abelian and $B=\mathbb{C}$, then a $(G,B)$-crossed product algebra is the same as a $G$-quantum torus in the terminology of \cite{Ne07b}. In this appendix we introduce a ``cohomology theory" for crossed product algebras, which is inspired by the classical cohomology theory of groups. The corresponding cohomology spaces are crucial for the classification of trivial noncommutative principal bundles with compact abelian structure group. A detailed discussion for the case $G=\mathbb{Z}^n$ can be found in \cite{Wa11a}.

\section*{Factor systems}\label{factor systems}

We start with associating so-called factor systems to pairs $(G,B)$ consisting of a group $G$ and a unital algebra $B$.

\begin{definition}\label{cochains}
Let $G$ be a group and $B$ be a unital algebra. 

(a) We write $C_B:B^{\times}\rightarrow\Aut(B)$ for the \emph{conjugation action} of $B^{\times}$ on $B$.

(b) We call a map $S\in C^1(G,\Aut(B))$ an \emph{outer action} of $G$ on $B$ if there exists 
\[\omega\in C^2(G,B^{\times})\,\,\,\text{with}\,\,\,\delta_S=C_B\circ\omega,
\]where $$\delta_S(g,g'):=S(g)S(g')S(gg')^{-1}.$$

(c) On the set of outer actions we define an equivalence relation by
\[S\sim S'\,\,\,\Leftrightarrow\,\,\,(\exists h\in C^1(G,B^{\times}))\,S'=(C_B\circ h)\cdot S
\]and call the equivalence class $[S]$ of an outer action $S$ a $G$-\emph{kernel}.

(d) For $S\in C^1(G,\Aut(B))$ and $\omega\in C^2(G,B^{\times})$ let
\[(d_S\omega)(g,g',g''):=S(g)(\omega(g',g''))\omega(g,g'g'')\omega(gg',g'')^{-1}\omega(g,g')^{-1}.
\]
\end{definition}

\begin{lemma}\label{action on factor system}
Let $n\in\mathbb{N}$ and $B$ be a unital algebra and consider the group $C^1(G,B^{\times})$ with respect to pointwise multiplication. This group acts on the set 
\[C^1(G,\Aut(B))\,\,\,\text{by}\,\,\,h.S:=(C_B\circ h)\cdot S
\]and on the product set
\begin{align}
C^1(G,\Aut(B))\times C^2(G,B^{\times})\,\,\,\text{by}\,\,\,h.(S,\omega):=(h.S,h\ast_S\omega)\notag
\end{align}
for
\[(h\ast_S\omega)(g,g'):=h(g)S(g)(h(g'))\omega(g,g')h(gg')^{-1}.
\]The stabilizer of $(S,\omega)$ is given by
\[C^1(G,B^{\times})_{(S,\omega)}=Z^1(G,Z(B)^{\times})_S
\]which depends only on $[S]$, but not on $\omega$, and the following assertions hold:
\begin{itemize}
\item[\emph{(a)}]
The subset 
\[\{(S,\omega)\in C^1(G,\Aut(B))\times C^2(G,B^{\times}):\delta_S=C_B\circ\omega\}
\]is invariant.
\item[\emph{(b)}] 
If $\delta_S=C_B\circ\omega$, then $\im(d_S\omega)\subseteq Z(B)^{\times}$ .
\item[\emph{(c)}]  
If $\delta_S=C_B\circ\omega$ and $h.(S.\omega)=(S',\omega')$, then $d_{S'}\omega'=d_S\omega$.
\end{itemize}
\end{lemma}

\begin{proof}
\,\,\,A proof of this Lemma can be found in \cite[Lemma 7.3.2]{Wa11b}.
\end{proof}




\begin{definition}\label{factor system}
Let $n\in\mathbb{N}$ and $B$ be a unital algebra. The elements of the set
\[Z^2(G,B):=\{(S,\omega)\in C^1(G,\Aut(B))\times C^2(G,B^{\times}):\delta_S=C_B\circ\omega, d_S\omega=1\}
\]are called \emph{factor systems} for the pair $(G,B)$. By Lemma \ref{action on factor system}, the set 
$Z^2(G,B)$ is invariant under the action of $C^1(G,B^{\times})$ and we write
\[H^2(G,B):=Z^2(G,B)/C^1(G,B^{\times})
\]for the corresponding cohomology space.
\end{definition}

\section*{Classification of crossed product algebras}
Given a group $G$ and a unital algebra $B$, the main goal of this section is to present a complete classification of $(G,B)$-crossed product algebras. We start with the precise definition:

\begin{definition}\label{algebraic trivial NCP T^n-bundles}(Crossed product algebras).
A $G$-graded unital algebra 
\[A=\bigoplus_{g\in G}A_g
\]with $B:=A_{1_G}$ is called an \emph{$(G,B)$-crossed product algebra}, if each grading space $A_g$ contains an invertible element.
\end{definition}

We now provide a construction that associates to each $(G,B)$-crossed product algebra $A$ a class in $H^2(G,B)$:

\begin{construction}\label{TNCT^nB ass class}(Characteristic classes).
Let $A$ be a $(G,B)$-crossed product algebra. The set
\[A^{\times}_h:=\bigcup_{g\in G}A^{\times}_g
\]of homogeneous units is a subgroup of $A^{\times}$ containing $B^{\times}$. We thus obtain an extension
\begin{align}
1\longrightarrow B^{\times}\longrightarrow A^{\times}_h\stackrel{q}\longrightarrow G\longrightarrow 1\label{formula 6.2}
\end{align}
of groups, where $q(a_g):=g$. In particular, $A^{\times}_h$ is equivalent to a crossed product of the form $B^{\times}\times_{(S,\omega)}G$ for a factor system $(S,\omega)\in Z^2(G,B)$. 
In fact, if we choose a section $\sigma:G\rightarrow A^{\times}_h$ of the extension (\ref{formula 6.2}) which is normalized in the sense that $\sigma(1_G)=1_A$, then we may endow the product set 
$B^{\times}\times G$ with the multiplication
\begin{align}
(b,g)(b',g')=(bS(g)(b')\omega(g,g'),gg'),\notag
\end{align}
where
\[S:G\rightarrow \Aut(B),\,\,\,S(g).b:=\sigma(g)b\sigma(g)^{-1}
\]and
\[\omega:G\times G\rightarrow B^{\times},\,\,\,(g,g')\mapsto \sigma(g)\sigma(g')\sigma(gg')^{-1}.
\]A short observations shows that this multiplication turns $B^{\times}\times G$ into a group and we write $B^{\times}\times_{(S,\omega)} G$ 
for the set $B^{\times}\times G$ endowed with the group multiplication induced by the factor system $(S,\omega)$. In particular, the map
\[B^{\times}\times_{(S,\omega)} G\rightarrow A^{\times}_h,\,\,\,(b,g)\mapsto b\sigma(g)
\]becomes an isomorphism of groups. In this way each $(G,B)$-crossed product algebra $A$ induces a \emph{characteristic class}
\[\chi(A):=[(S,\omega)]\in H^2(G,B).
\]
\end{construction}

\begin{lemma}
Each $(G,B)$-crossed product algebra $A$ possesses a characteristic class $\chi(A)\in H^2(G,B)$.
\end{lemma}

\begin{proof}
\,\,\,Indeed, this statement immediately follows from Construction \ref{TNCT^nB ass class}.
\end{proof}

\begin{definition}\label{Equivlance}
Two $(G,B)$-crossed product algebras $A$ and $A'$ are called \emph{equivalent} if there is an algebra isomorphism
$\varphi:A\rightarrow A'$ satisfying $\varphi(A_g)=A'_g$ for all $g\in G$ and $\varphi_{\mid B}=\id_B$. If $A$ and $A'$ are equivalent $(G,B)$-crossed product algebras, then we write $[A]$ for the corresponding equivalence class.
\end{definition}

\begin{proposition}\label{equ  ATNCT^nB same class}
Let $A$ and $A'$ be two equivalent $(G,B)$-crossed product algebras. Then their corresponding characteristic classes coincide, i.e.,
\[\chi(A)=\chi(A')\in H^2(G,B).
\]
\end{proposition}

\begin{proof}
\,\,\,If $A$ and $A'$ are equivalent, then the same holds for their corresponding extensions of Construction \ref{TNCT^nB ass class}. Thus, the claim follows from \cite[Chapter IV, Section 4]{Ma95}.
\end{proof}

\begin{definition}(Set of equivalence classes).\label{ext shit}
Let $G$ be a group and $B$ be a unital algebra. We write $\Ext(G,B)$ for the set of all equivalence classes of $(G,B)$-crossed product algebras.
\end{definition}


\begin{lemma}\label{chi well-defined}
The map
\[\chi:\Ext(G,B)\rightarrow H^2(G,B),\,\,\,[A]\mapsto\chi(A)
\]is well-defined.
\end{lemma}

\begin{proof}
\,\,\,The statement immediately follows from Proposition \ref{equ  ATNCT^nB same class}.
\end{proof}

\vspace*{0,3cm}

In the remaining part we show that the map $\chi$ is a bijection:

\begin{construction}\label{realization of TNCT^B from factor systems I}
Let $G$ be a group and $B$ be a unital algebra. Further, let 
\[A:=\bigoplus_{g\in G}Bv_{g}
\]be a vector space with basis $(v_g)_{g\in G}$. For a factor system $(S,\omega)\in Z^2(G,B)$, we
define a multiplication map 
\[m_{(S,\omega)}:A\times A\rightarrow A
\]given on homogeneous elements by
\begin{align}
m_{(S,\omega)}(bv_g,b'v_{g'}):=b(S(g)(b'))\omega(g,g')v_{gg'},\label{multiplication}
\end{align}
and write $A_{(S,\omega)}$ for the vector space $A$ endowed with the multiplication (\ref{multiplication}). A short calculation shows that $A_{(S,\omega)}$ is a $G$-graded unital algebra with $A_{1_G}=B$ and unit $v_{1_G}$. Moreover, each grading space $A_g$ contains invertible elements w.r.t. to this multiplication: Indeed, if $b\in B^{\times}$, then the inverse of $bv_g$ is given by
\[S(g)^{-1}(b^{-1}\omega(g,g^{-1})^{-1})v_{g^{-1}}.
\]Thus, $A_{(S,\omega)}$ is a $(G,B)$-crossed product algebra which has characteristic class \mbox{$\chi(A_{(S,\omega)})=[(S,\omega)]$}.
\end{construction}

\begin{proposition}\label{realization of TNCT^B from factor systems II}
If $G$ is a group and $B$ a unital algebra, then each element $[(S,\omega)]\in H^2(G,B)$ can be realized by a $(G,B)$-crossed product algebra with $\chi(A)=[(S,\omega)]$.
\end{proposition}

\begin{proof}
\,\,\,This statement is a consequence of Construction \ref{realization of TNCT^B from factor systems I}: In fact, if $[(S,\omega)]$ represents a class in $H^2(G,B)$, then $A_{(S,\omega)}$ satisfies the requirements of the proposition.
\end{proof}

\begin{proposition}\label{realization of TNCT^B from factor systems III}
Let $A$ be a $(G,B)$-crossed product algebra. Then $A$ is equivalent to a $(G,B)$-crossed product algebra of the form $A_{(S,\omega)}$ for some factor system $(S,\omega)\in Z^2(G,B)$.
\end{proposition}

\begin{proof}
\,\,\,Let $A$ be a $(G,B)$-crossed product algebra. We consider the corresponding short exact sequence of groups
\[1\longrightarrow A^{\times}_{1_G}\longrightarrow A^{\times}_h\longrightarrow G\longrightarrow 1
\]and choose a section $\sigma:G\rightarrow A^{\times}_h$. Now, a short calculation shows that the map
\[\varphi:(A=\bigoplus_{g\in G}A_{g},m_A)\rightarrow\big(\bigoplus_{g\in G}Bv_{g},m_{(C_B\circ\sigma,\delta_{\sigma})}\big),
\]given on homogeneous elements by
\[\varphi(a_g):=a_{g}\sigma(g)^{-1}v_{g},
\]defines an equivalence of $(G,B)$-crossed product algebras. 
\end{proof}

\begin{proposition}\label{realization of TNCT^B from factor systems IV}
Let $G$ be a group and $B$ be a unital algebra. Further, let $(S,\omega)$ and $(S',\omega')$ be two factor systems in $ Z^2(G,B)$ with $[(S,\omega)]=[(S',\omega')]$. Then the corresponding $(G,B)$-crossed product algebras $A_{(S,\omega)}$ and $A_{(S',\omega')}$ are equivalent.
\end{proposition}

\begin{proof}
\,\,\,First recall that the condition $[(S',\omega')]=[(S,\omega)]$ is equivalent to the existence of an element $h\in C^1(G,B^{\times})$ with 
\[h.(S,\omega)=(S',\omega'),
\]Now, a short observation shows that the map 
\[\varphi:\big(\bigoplus_{g\in G}Bv_{g},m_{(S',\omega')}\big)\rightarrow\big(\bigoplus_{g\in G}Bv_{g},m_{(S,\omega)}\big),
\]given on homogeneous elements by
\[\varphi(bv_{g})=bh(g)v_{g},
\]is an automorphism of vector spaces leaving the grading spaces invariant. We further have
\begin{align}
m_{(S,\omega)}(\varphi(bv_{g}),\varphi(b'v_{g'}))&=m_{(S,\omega)}(bh(g)v_{g},b'h(g')v_{g'})\notag\\
&=bh({g})S({g})(b'h(g'))\omega({g},g')v_{gg'}\notag\\
&=b[C_B(h({g}))(S({g})(b'))]h({g})S({g})(h(g'))\omega({g},g')v_{gg'}\notag\\
&=b(h.S)({g})(b')(h\ast_S\omega)({g},g')h({gg'})v_{gg'}\notag\\
&=\varphi(b(h.S)({g})(b')(h\ast_S\omega)({g},{g'})v_{gg'})\notag\\
&=\varphi(m_{(S',\omega')}(bv_{g},b'v_{g'})).\notag
\end{align}
Hence, the map $\varphi$ actually defines an equivalence of $(G,B)$-crossed product algebras. 
\end{proof}

\vspace*{0,3cm}

We are now ready to state and proof the main theorem of this section:

\begin{theorem}\label{classification on TNCT^B}
Let $G$ be a group and $B$ be a unital algebra. Then the map
\[\chi:\Ext(G,B)\rightarrow H^2(G,B),\,\,\,[A]\mapsto\chi(A)
\]is a well-defined bijection.
\end{theorem}

\begin{proof}
\,\,\,We first note that Lemma \ref{chi well-defined} implies that the map $\chi$ is well-defined. The surjectivity of $\chi$ follows from Proposition \ref{realization of TNCT^B from factor systems II}. Hence, it remains to show that the map $\chi$ is injective: For this let $A$ and $A'$ be two $(G,B)$-crossed product algebras with \mbox{$\chi(A)=\chi(A')$}. By Proposition \ref{realization of TNCT^B from factor systems III} we may assume that $A=A_{(S,\omega)}$ and $A'=A_{(S',\omega')}$ for two factor systems $(S,\omega)$ and $(S',\omega')$ in $ Z^2(G,B)$ with $[(S',\omega')]=[(S,\omega)]$. Thus, the claim follows from Proposition \ref{realization of TNCT^B from factor systems IV}.
\end{proof}

\section*{$G$-Kernels} 

We just saw that the set of all equivalence classes of $(G,B)$-crossed product algebras is classified by the cohomology space $H^2(G,B)$. Moreover, Proposition \ref{equ  ATNCT^nB same class} in particular implies that equivalent $(G,B)$-crossed product algebras correspond to the same $G$-kernel (cf. Definition \ref{cochains} (c)). This leads to the following definition:

\begin{definition}(Equivalence classes of $G$-kernels).
Let $G$ be a group and $B$ be a unital algebra. We write $\Ext(G,B)_{[S]}$ for the set of equivalence classes of $(G,B)$-crossed product algebras corresponding to the $G$-kernel $[S]$. 
\end{definition}

Note that the set $\Ext(G,B)_{[S]}$ may be empty. The aim of this section is to show how to classify this set and give conditions for its non-emptiness. The following proposition basically states that if $\Ext(G,B)_{[S]}$ is non-empty, then it is classified by the second group cohomolgy space $H^2(G,Z(B)^{\times})_{[S]}$ (cf. \cite[Chapter IV, Section 4]{Ma95} for the necessary definitions on classical group cohomology.)

\begin{proposition}\label{class of Z-kernels I}
Let $G$ be a group and $B$ be a unital algebra. Further, let $S$ be an outer action of $G$ on $B$ with $\Ext(G,B)_{[S]}\neq\emptyset$. Then the following assertions hold:
\begin{itemize}
\item[\emph{(a)}]
Each extension class in $\Ext(G,B)_{[S]}$ can be represented by $(G,B)$-crossed product algebra of the form $A_{(S,\omega)}$.
\item[\emph{(b)}]
Any other $(G,B)$-crossed product algebra $A_{(S,\omega')}$ representing an element of $\Ext(G,B)_{[S]}$ satisfies
\[\omega'\cdot\omega^{-1}\in Z^2(G,Z(B)^{\times})_{[S]}.
\]
\item[\emph{(c)}]
Two $(G,B)$-crossed product algebras $A_{(S,\omega)}$ and $A_{(S,\omega')}$ are equivalent if and only if
\[\omega'\cdot\omega^{-1}\in B^2(G,Z(B)^{\times})_{[S]}.
\]
\end{itemize}
\end{proposition}

\begin{proof}
\,\,\,(a) From Proposition \ref{realization of TNCT^B from factor systems II} we know that each $(G,B)$-crossed product algebra $A$ is equivalent to one of the form $A_{(S',\omega')}$. If $[S]=[S']$ and the element $h\in C^1(G,B^{\times})$ satisfies $h.S=S'$, then $h^{-1}.(S',\omega')=(S,h^{-1}\ast_{S'}\omega')$, so that $\omega'':=h^{-1}\ast_{S'}\omega'$ implies that $A_{(S',\omega')}$ and $A_{(S,\omega'')}$ are equivalent. This means that each extension class in $\Ext(G,B)_{[S]}$ can be represented by a $(G,B)$-crossed product algebra of the form $A_{(S,\omega'')}$.

(b) If $(S,\omega)$ and $(S,\omega')$ are two factor systems for the pair $(G,B)$, then $C_B\circ\omega=\delta_S=C_B\circ\omega'$ implies $\beta:=\omega'\cdot\omega^{-1}\in C^2(G,Z(B)^{\times})$. Further $d_S\omega'=d_S\omega=1_B$, so that 
\[1_B=d_S\omega'=d_S\omega\cdot d_S\beta=d_S\beta
\]implies $\beta\in Z^2(G,Z(B)^{\times})_{[S]}$ and hence the assertion. 

(c) ($``\Rightarrow"$) In view of Proposition \ref{equ  ATNCT^nB same class} and Proposition \ref{realization of TNCT^B from factor systems IV}, the equivalence of the $(G,B)$-crossed product algebras $A_{(S,\omega)}$ and $A_{(S,\omega')}$is equivalent to the existence of an element $h\in C^1(G,B^{\times})$ with
\[S=h.S=(C_B\circ h)\cdot S\,\,\,\text{and}\,\,\,\omega'=h\ast_S\omega.
\]Then $C_B\circ h=\id_B$ implies that $h\in C^1(G,Z(B)^{\times})$ which in turn immediately leads to $\omega'=h\ast_S\omega=(d_Sh)\cdot\omega$, i.e., 
\[\omega'\omega^{-1}\in B^2(G,Z(B)^{\times})_{[S]}.
\]

($``\Leftarrow"$) If conversely $\omega'\omega^{-1}=d_Sh$ for some $h\in C^1(G,Z(B)^{\times})$, then we easily conclude that $h.S=S$ and $\omega'=h\ast_S\omega$.
\end{proof}

\begin{theorem}\label{class of Z-kernels II}
Let $G$ be a group and $B$ be a unital algebra. Further, let $[S]$ be a $G$-kernel with $\Ext(G,B)_{[S]}\neq\emptyset$. Then the map
\[H^2(G,Z(B)^{\times})_{[S]}\times\Ext(G,B)_{[S]}\rightarrow\Ext(G,B)_{[S]}
\]given by
\[([\beta],[A_{(S,\omega)}])\mapsto[A_{(S,\omega\cdot\beta)}]
\]is a well-defined action which is both transitive and free.
\end{theorem}

\begin{proof}
\,\,\,This follows directly form Proposition \ref{class of Z-kernels I}.
\end{proof}

\begin{remark}\label{B abelian}(Commutative fixed point algebras). Let $G$ be a group and suppose that $B$ is a commutative unital algebra. Then the adjoint representation of $B$ is trivial and a factor system $(S,\omega)$ for $(G,B)$ consists of a module structure given by a homomorphism $S:G\rightarrow \Aut(B)$ and an element $\omega\in C^2(G,B^{\times})$ satisfying \mbox{$d_S\omega=1_B$} (which is equivalent to $\omega\in Z^2(G,B^{\times})_{[S]}$). In this case we simply write 
$A_{\omega}$ for the corresponding $(G,B)$-crossed product algebra. 

Further $S\sim S'$ if and only if $S=S'$. Hence a $G$-kernel $[S]$ is the same as a $G$-module structure $S$ on $B$ and $\Ext(G,B)_S:=\Ext(G,B)_{[S]}$ is the set of all $(G,B)$-crossed product algebras corresponding to the $G$-module structure on $B$ given by $S$.

According to Theorem \ref{class of Z-kernels II}, these equivalence classes correspond to cohomology classes of cocycles, so that the map
\[H^2(G,B^{\times})_S\rightarrow\Ext(G,B)_S,\,\,\,[\omega]\mapsto[A_{\omega}]
\]is a well-defined bijection.
\end{remark}

We now give a condition that ensures the non-emptiness of the set $\Ext(G,B)_{[S]}$ for a given $G$-kernel $S$. We first need the following definition:

\begin{definition}\label{char. class d_Sw}
Let $G$ be a group and $B$ be a unital algebra. Further, let $S$ be an outer action of $G$ on $B$ and choose $\omega\in C^2(G,B^{\times})$ with $\delta_S=C_B\circ\omega$. The cohomology class
\[\nu(S):=[d_S\omega]\in H^3(G,Z(B)^{\times})_S
\]does not depend on the choice of $\omega$ and is constant on the equivalence class of $S$, so that we may also write $\nu([S]):=\nu(S)$ (cf. \cite[Corollary 7.3.25 for $G=\mathbb{Z}^n$]{Wa11b}). We call $\nu(S)$ the \emph{characteristic class} of $S$.
\end{definition}

The next theorem provides a group theoretic criterion for the non-emptiness of the set $\Ext(G,B)_{[S]}$:

\begin{theorem}\label{Ext(Z^n,B)_S non-empty}
Let $G$ be a group and $B$ be a unital algebra. If $[S]$ is a $G$-kernel, then
\[\nu([S])={\bf 1}\,\,\,\Leftrightarrow\,\,\,\Ext(G,B)_{[S]}\neq\emptyset.
\]
\end{theorem}

\begin{proof}
\,\,\,($``\Leftarrow"$) If there exists a$(G,B)$-crossed product algebra $A$ corresponding to $[S]$, then Proposition \ref{realization of TNCT^B from factor systems III} implies that we may w.l.o.g. assume that it is of the form $A_{(S,\omega)}$. In particular, this means that $d_S\omega=1_B$. Hence, we obtain $\nu([S])=[d_S\omega]=\textbf{1}$.

($``\Rightarrow"$) Suppose, conversely, that $\nu([S])=[d_S\omega]=\textbf{1}$. Then there exists an element $\omega\in C^2(G,B^{\times})$ with $\delta_S=C_B\circ\omega$ and some $h\in C^2(G,Z(B)^{\times})$ with $d_S\omega=d_Sh^{-1}$. Therefore $\omega':=\omega\cdot h\in C^2(G,B^{\times})$ satisfies 
\[d_S\omega'=d_S\omega\cdot d_Sh=1\,\,\,\text{and}\,\,\,\delta_S=C_B\circ\omega'.
\]Hence, $(S,\omega')$ is a factor system for $(G,B)$ and Proposition \ref{realization of TNCT^B from factor systems II} implies the existence of a $(G,B)$-crossed product algebra $A_{(S,\omega')}$ corresponding to the $G$-kernel $[S]$.
\end{proof}

\section*{A Landstad duality theorem for crossed product algebras}

In the following let $G$ be a compact abelian group. In the last part of this appendix we present a Landstad duality theorem for $C^*$-dynamical systems $(A,G,\alpha)$ with the property that each isotypic component contains an invertible element. We did not find such a result in the literature.

\begin{notation}
Given a $C^*$-dynamical system $(A,G,\alpha)$, we write 
\[A_{\varphi}:=\{a\in A:\,(\forall g\in G)\,\alpha(g,a)=\varphi(g)\cdot a\}
\]for the isotypic component corresponding to an element $\varphi$ in the dual group $\widehat{G}$ of $G$. In particular, we write $B:=A^G=A_{\bf 0}$ for the corresponding fixed point algebra. 
\end{notation}

\begin{lemma}\label{landstad I}
Let $G$ be a compact abelian group and $(A,G,\alpha)$ be a $C^*$-dynamical system such that each isotypic component contains an invertible element. Then each isotypic component also contains a unitary element.
\end{lemma}

\begin{proof}
\,\,\,Let $\varphi\in\widehat{G}$ and $a_{\varphi}\in A_{\varphi}$ be an invertible element. Then we conclude from polar decomposition for $C^*$-algebras that $a_{\varphi}=\vert a_{\varphi}\vert u_{\varphi}$ for some positive element $\vert a_{\varphi}\vert$ and a unitary element $u_{\varphi}$. Further, the construction implies that $\vert a_{\varphi}\vert\in B$ and thus that $u_{\varphi}\in A_{\varphi}$ as desired. Hence, each isotypic component contains a unitary element.
\end{proof}

\begin{construction}\label{landstad II}
Let $G$ be a compact abelian group and $(A,G,\alpha)$ be a $C^*$-dynamical system such that each isotypic component contains an invertible element. According to Lemma \ref{landstad I}, we may choose in each isotypic component $A_{\varphi}$ a unitary element $u_{\varphi}$, which in turn leads to a factor system $(S,\omega)$ for the pair $(\widehat{G},B)$ (cf. \ref{factor system}) defined by

maps
\[S:\widehat{G}\rightarrow\Aut(B),\,\,\,S(\varphi).b:=u_{\varphi}bu_{\varphi}^*
\]and
\[\omega:\widehat{G}\times\widehat{G}\rightarrow\U(B),\,\,\,\omega(\varphi,\varphi'):=u_{\varphi}u_{\varphi'}u_{\varphi\cdot\varphi'}^*,
\]where $\U(B)$ denotes the unitary group of $B$. The corresponding algebra $A_{(S,\omega)}$ from Construction \ref{realization of TNCT^B from factor systems I} carries a natural involution given by
\[(bv_{\varphi})^{*}:=\omega(-\varphi,\varphi)^{-1}S(-\varphi).b^*v_{-\varphi}
\]and a short observation shows that multiplication and involution are continuous for the $\ell^1$-norm 
\[\Vert (b_{\varphi}v_{\varphi})_{\varphi\in\widehat{G}}\Vert_1:=\sum_{\varphi\in\widehat{G}}\Vert b_{\varphi}\Vert.
\]We write $\ell^1(A_{(S,\omega)})$ for the corresponding involutive Banach algebra and further $C^*(A_{(S,\omega)})$ for the enveloping $C^*$-algebra of $\ell^1(A_{(S,\omega)})$. 
\end{construction}

\begin{remark}
In view of the terminology of Example \ref{crossed product as example again} (cf. Example \ref{c-stern crossed product again}) we could also write $\ell^1(B\times_{(S,\omega)}\widehat{G})$ for $\ell^1(A_{(S,\omega)})$ and $C^*(B\times_{(S,\omega)}\widehat{G})$ for $C^*(A_{(S,\omega)})$.
\end{remark}

\begin{proposition}
Suppose we are in the situation of Construction \ref{landstad II}. Then $A$ is isomorphic \emph{(}as $C^*$-algebra\emph{)} to $C^*(A_{(S,\omega)})$.
\end{proposition}

\begin{proof}
\,\,\,The assertion immediately follows by extending the (algebraic) isomorphism of Proposition \ref{realization of TNCT^B from factor systems III} to an $C^*$-algebraic isomorphism of the corresponding completions. Alternatively, the assertion follows from \cite[Proposition 4.2]{Ex96}.

\end{proof}

\section*{Acknowledgements} 
We thank Siegfried Echterhoff, Karl-Hermann Neeb and Ryszard Nest for fruitful discussions on this topic. Moreover, we would like to thank the anonymous referee for useful comments.

\end{document}